\documentclass[pdflatex,sn-mathphys]{sn-jnl}

\theoremstyle{thmstyleone}%
\newtheorem{thm}{Theorem}[section]

\theoremstyle{thmstyletwo}%
\newtheorem{re}{Remark}[section]%

\theoremstyle{thmstylethree}%
\newtheorem{defi}{Definition}[section]%
\newtheorem{ass}{Assumption}[section]
\newtheorem{lem}{Lemma}[section]
\numberwithin{equation}{section}

\usepackage{subfigure}

\begin{document}
	
	\title[FORMDA for stochastic nonconvex-concave minimax problems]{An accelerated first-order regularized momentum descent ascent algorithm for stochastic nonconvex-concave minimax problems}
	%
	%
	%
	
	
	\author[1]{\sur{Huiling Zhang}}
	\author*[1,2]{\sur{Zi Xu}\email{xuzi@shu.edu.cn}}
	
	
	\affil[1]{\orgdiv{Department of Mathematics, College of Sciences}, \orgname{Shanghai University}, \orgaddress{\city{Shanghai}, \postcode{200444}, \country{P.R.China}}}
	\affil[2]{\orgdiv{Newtouch Center for Mathematics of Shanghai University}, \orgname{Shanghai University}, \orgaddress{\city{Shanghai}, \postcode{200444}, \country{P.R.China}}}
	
	
	
	\abstract{Stochastic nonconvex minimax problems have attracted wide attention in machine learning, signal processing and many other fields in recent years. In this paper, we propose an accelerated first-order regularized  momentum descent ascent algorithm (FORMDA) for solving stochastic nonconvex-concave minimax problems. The iteration complexity of the algorithm is proved to be $\tilde{\mathcal{O}}(\varepsilon ^{-6.5})$ to obtain an $\varepsilon$-stationary point, which achieves the best-known complexity bound for single-loop algorithms to solve the stochastic nonconvex-concave minimax problems under the stationarity of the objective function.}
	

	
	
	
	\maketitle
	
	\section{Introduction}\label{sec1}
	Consider the following stochastic nonconvex-concave minimax optimization problem:
	\begin{equation}\label{prob-s}
		\min \limits_{x\in \mathcal{X}}\max \limits_{y\in \mathcal{Y}}g(x,y)=\mathbb{E}_{\zeta\sim D}[G(x,y,\zeta)],
	\end{equation}
	where $\mathcal{X}\subseteq \mathbb{R}^{d_x}$ and $\mathcal{Y}\subseteq \mathbb{R}^{d_y}$ are nonempty convex and compact sets, $G(x,y,\zeta): \mathcal{X} \times \mathcal{Y} \rightarrow \mathbb{R}$ is a smooth function, possibly nonconvex in $x$ and concave in $y$, $\zeta$ is a random variable following an unknown distribution $D$, and $\mathbb{E}$ denotes the expectation function. We focus on the single-loop first-order algorithms to solve problem \eqref{prob-s}. This problem has attracted increasing attention in machine learning, signal processing, and many other research fields in recent years, e.g., distributed nonconvex optimization \cite{Giannakis,Liao,Mateos}, wireless system \cite{Chen20}, statistical learning \cite{Abadeh,Giordano}, etc.
	
	There are few existing first-order methods for solving  stochastic nonconvex-concave minimax optimization problem. Rafique et al. \cite{Rafique} proposed a proximally guided stochastic mirror descent method (PG-SMD), which updates $x$ and $y$ simultaneously, and the iteration complexity has been proved to be $\tilde{\mathcal{O}}(\varepsilon ^{-6})$ to an approximate $\varepsilon$-stationary point of  $\Phi (\cdot) = \max_{y\in \mathcal{Y}} f(\cdot, y)$. Zhang et al. \cite{Zhang2022} proposed a SAPD+ which achieves the oracle complexities of $\mathcal{O}\left(\varepsilon^{-6} \right)$ for solving deterministic nonconvex-concave minimax problems. Both algorithms are multi-loop algorithms. On the other hand, there are few existing single-loop algorithms for solving stochastic nonconvex-concave minimax problems.
	Lin et al. \cite{Lin2019} proposed a stochastic gradient descent ascent (SGDA) and  Bot et al. \cite{Bot} proposed a stochastic alternating GDA, both of them require $\mathcal{O}\left(\varepsilon^{-8} \right)$ stochastic gradient evaluations to obtain an approximate $\varepsilon$-stationary point of  $\Phi (\cdot) = \max_{y\in \mathcal{Y}} f(\cdot, y)$. Boroun et al. \cite{Boroun} proposed a novel single-loop
	stochastic primal-dual algorithm with momentum (SPDM) for a special class of nonconvex-concave minimax problems, i.e., the objective function statisfies the Polyak-{\L}ojasiewicz (PL) condition with respect to $x$, and the iteration complexity is proved to be $\mathcal{O}\left(\varepsilon^{-4} \right)$.
	
	There are some exisiting first order methods on solving stochastic nonconvex-strongly concave minimax optimization problems. 
	Lin et al. \cite{Lin2019} proposed a stochastic gradient descent ascent (SGDA) which requires $\mathcal{O}\left( \kappa^{3}\varepsilon^{-4} \right)$ stochastic gradient evaluations to get an $\varepsilon$-stationary point. Luo et al. \cite{luo2020stochastic}  proposed a stochastic recursive gradient descent ascent (SREDA) algorithm, which requires a total number of $\mathcal{O}\left( \kappa^{3}\varepsilon^{-3} \right)$ stochastic gradient evaluations.  
	Huang et al. \cite{Huang} proposed an accelerated first-order momentum descent ascent (Acc-MDA) method with the total number of stochastic gradient evaluations being  $\tilde{\mathcal{O}}\left( \kappa^{4.5}\varepsilon^{-3} \right)$. 
	
	There are also some zeroth-order methods for stochastic nonconvex-strongly concave minimax optimization problems. Liu et al. \cite{Liu} proposed a ZO-Min-Max algorithm and it needs $\mathcal{O}\left( \kappa^{6}\varepsilon^{-6} \right)$ function evaluations to obtain $\varepsilon$-stationary point.  Wang et al. \cite{Wang} proposed a zeroth-order gradient descent multi-step ascent (ZO-GDMSA) algorithm, and the function evaluation complexity of $\tilde{\mathcal{O}}\left( \kappa^{2}\varepsilon^{-4} \right)$ is proved. Luo et al. \cite{luo2020stochastic}  proposed a stochastic recursive gradient descent ascent (SREDA) algorithm, and its function evaluation complexity is $\mathcal{O}\left( \kappa^{3}\varepsilon^{-3} \right)$. Xu et al. \cite{Xu20} proposed a zeroth-order variance reduced gradient descent ascent (ZO-VRGDA) algorithm with $\mathcal{O}\left( \kappa^{3}\varepsilon^{-3} \right)$ function evaluation complexity. Huang et al. \cite{Huang} proposed an accelerated zeroth-order momentum descent ascent (Acc-ZOMDA) method which owns the function evaluation complexity of $\tilde{\mathcal{O}}\left( \kappa^{4.5}\varepsilon^{-3} \right)$. 
	
	There are few exisiting methods on solving stochastic nonconvex-nonconcave minimax optimization problems. 
	Yang et al. \cite{Yang22} proposed a smoothed GDA algorithm for nonconvex-PL minimax problems with the iteration complexity of $\mathcal{O}(\varepsilon^{-4})$. Xu et al. \cite{xu22zeroth} proposed a zeroth-order variance reduced alternating gradient descent ascent (ZO-VRAGDA) algorithm for solving nonconvex-PL minimax problems, which can obtain the total complexity of $\mathcal{O}(\varepsilon^{-3})$. Huang \cite{Huang23} proposed two enhanced momentum-based gradient descent ascent methods for nonconvex-PL minimax problems, which owns the iteration complexity of $\tilde{\mathcal{O}}(\varepsilon^{-3})$.
	
	
	In this paper, we propose an accelerated first-order regularized  momentum descent ascent algorithm (FORMDA) for solving stochastic nonconvex-concave minimax problems. The iteration complexity of the algorithm is proved to be $\tilde{\mathcal{O}}(\varepsilon ^{-6.5})$ to obtain an $\varepsilon$-stationary point.
	
	\subsection{Related Works.}
	We give a brief review on algorithms for solving deterministic minimax optimization problems. There are many existing works on solving convex-concave minimax optimization problems. Since we focus on nonconvex minimax problems, we do not attempt to survey it in this paper and refer to \cite{Chen,Lan2016,Mokhtari,Nemirovski,Nesterov2007,Ou15,Ou21,Tom,Zhang22} for more details.
	
	For deterministic nonconvex-concave minimax problem, there are two types of algorithms, i.e., multi-loop algorithms and single-loop algorithms. One intensively studied type is nested-loop algorithms. Various algorithms of this type have been proposed in \cite{Rafique,nouiehed2019solving,Thek19,Kong,ostrovskii2020efficient}. Lin et al. \cite{lin2020near} proposed a class of accelerated algorithms for smooth nonconvex-concave minimax problems with the complexity bound of $\tilde{\mathcal{O}}\left( \varepsilon ^{-2.5} \right)$ which owns the best iteration complexity till now. On the other hand, fewer studies focus on single-loop algorithms for nonconvex-concave minimax problems. 
	One typical method is the gradient descent-ascent (GDA) method, which performs a gradient descent step on $x$ and a gradient ascent step on $y$ simultaneously at each iteration. Jin et al. \cite{Lin2019} proposed a GDmax algorithm with iteration complexity of $\tilde{\mathcal{O}}\left( \varepsilon ^{-6} \right) $. Lu et al.~\cite{Lu} proposed a hybrid block successive approximation (HiBSA) algorithm, which can obtain an $\varepsilon$-stationary point of $f(x,y)$ in $\tilde{\mathcal{O}}\left( \varepsilon ^{-4} \right)$ iterations. Pan et al. \cite{Pan} proposed a new alternating gradient projection algorithm for nonconvex-linear minimax problem with the iteration complexity of $\mathcal{O}\left( \varepsilon ^{-3} \right)$.
	Xu et al. \cite{Xu} proposed a unified single-loop alternating gradient projection (AGP) algorithm for solving nonconvex-(strongly) concave and (strongly) convex-nonconcave minimax problems, which can find an $\varepsilon$-stationary point with the gradient complexity of $\mathcal{O}\left( \varepsilon ^{-4} \right)$. 
	Zhang et al. \cite{Zhang} proposed a smoothed GDA algorithm which achieves $\mathcal{O}\left( \varepsilon ^{-4} \right)$ iteration complexity for general nonconvex-concave minimax problems. 
	
	Few algorithms have been proposed for solving more general deterministic nonconvex-nonconcave minimax problems. Sanjabi et al. \cite{Sanjabi18} proposed a multi-step gradient descent-ascent algorithm which can find an $\varepsilon$-first order Nash equilibrium in $\mathcal{O}(\varepsilon^{-2}\log\varepsilon^{-1})$ iterations when a one-sided PL condition is satisfied. Yang et al. \cite{Yang} showed the alternating GDA algorithm converges globally at a linear rate for a subclass of nonconvex-nonconcave objective functions satisfying a two-sided PL inequality. Song et al. \cite{SongOptDE} proposed an optimistic dual extrapolation (OptDE) method with the iteration complexity of $\mathcal{O}\left( \varepsilon^{-2} \right)$, if a weak solution exists. Hajizadeh et al. \cite{Hajizadeh} showed that a damped version of extra-gradient method linearly converges to an $\varepsilon$-stationary point of nonconvex-nonconcave minimax problems that the nonnegative interaction dominance condition is satisfied. Xu et al. \cite{xu22zeroth} proposed a zeroth-order alternating gradient descent ascent (ZO-AGDA)
	algorithm for solving NC-PL minimax problem, which can obtain the iteration complexity $\mathcal{O}(\varepsilon^{-2})$. 
	For more related results, we refer to \cite{Bohm,cai2022accelerated,DiakonikolasEG,Doan,Grimmer,Jiang,Lee}.

	\textbf{Notations}
	For vectors, we use $\|\cdot\|$ to represent the Euclidean norm and its induced matrix norm; $\left\langle x,y \right\rangle$ denotes the inner product of two vectors of $x$ and $y$. We use $\nabla_{x} f(x,y)$ (or $\nabla_{y} f(x, y)$) to denote the partial derivative of $f(x,y)$ with respect to $x$ (or $y$) at point $(x, y)$, respectively.
	We use the notation $\mathcal{O} (\cdot)$ to hide only absolute constants which do not depend on any problem parameter, and $\tilde{\mathcal{O}}(\cdot)$ notation to hide only absolute constants and log factors. A continuously differentiable function $f(\cdot)$ is called $L$-smooth if there exists a constant $L> 0$ such
	that for any given $x,y \in \mathcal{X}$,
	\begin{equation*}
		\|\nabla f(x)-\nabla f(y)\|\le L\|x-y\|.
	\end{equation*}
	A continuously differentiable function $f(\cdot)$ is called $\mu$-strongly convcave if there exists a constant $\mu> 0$ such
	that for any $x,y \in \mathcal{X}$,
	\begin{equation*}
		f(y)\le f(x)+\langle \nabla f(x),y-x \rangle-\frac{\mu}{2}\|y-x\|^2.
	\end{equation*}
	
	The paper is organized as follows. In Section \ref{sec2}, an accelerated first-order momentum projection gradient algorithm is proposed for stochastic nonconvex-concave minimax problem, and its iteration complexity is also established. Numerical results are presented in Section \ref{sec4} to show the efficiency of the proposed algorithm. Some conclusions are made in the last section.

	\section{An Accelerated First-order Regularized  Momentum Descent Ascent Algorithm}\label{sec2}
	In this section, based on the framework of the Acc-MDA algorithm~\cite{Huang}, we propose an accelerated first-order regularized  momentum descent ascent algorithm (FORMDA) for solving problem \eqref{prob-s}. 
	At the $k$th iteration of FORMDA, we consider a regularized function of $g(x, y)$, i.e.,
	\begin{equation}\label{sec2:gk}
		g_k(x,y)=g(x,y)-\frac{\rho_k}{2}\|y\|^2,
	\end{equation}
	where $\rho_k\ge0$ is a regularization parameter. Compared to the Acc-MDA algorithm, the main difference in the FORMDA algorithm is that instead of $g(x,y)$, the gradient of $g_k(x,y)$, is computed and used at each iteration. More detailedly, at the $k$th iteration, for some given $I=\{\zeta_1,\cdots,\zeta_b\}$ drawn i.i.d. from an unknown distribution, by denoting
	\begin{align}
		\tilde{G}_k(x,y;\zeta_j)&=G(x,y;\zeta_j)-\frac{\rho_k}{2}\|y\|^2,\label{sec2:Gk}
	\end{align}
	we compute the gradient of the stochastic function $\tilde{G}_k(x,y;I)$ as follows,
	\begin{align}
		\nabla_x\tilde{G}_k(x,y;I)&=\frac{1}{b}\sum_{j=1}^{b}\nabla_x\tilde{G}_k(x,y;\zeta_j),\label{sec2:Gkx}\\
		\nabla_y\tilde{G}_k(x,y;I)&=\frac{1}{b}\sum_{j=1}^{b}\nabla_y\tilde{G}_k(x,y;\zeta_j).\label{sec2:Gky}
	\end{align}
	Then, based on $\nabla_x\tilde{G}_k(x,y;I)$ and $\nabla_y\tilde{G}_k(x,y;I)$, we compute the variance-reduced stochastic gradient $v_k$ and $w_k$ as shown in \eqref{vk} and \eqref{wk} respectively with $0< \gamma_k \leq 1$ and $0<\theta_{k}\leq 1$ that will be defined later. We update $x_k$ and $y_k$ through alternating stochastic gradient projection with the momentum technique shown in \eqref{sfo1:update-tx}-\eqref{sfo1:update-y}, which is similar to that in the Acc-MDA algorithm. The proposed FORMDA algorithm is formally stated in Algorithm \ref{sfoalg:1}.

	\begin{algorithm}[t]
		\caption{An Accelerated First-order Regularized  Momentum Descent Ascent Algorithm (FORMDA)}
		\label{sfoalg:1}
		\begin{algorithmic}
			\State{\textbf{Step 1}:Input $x_1,y_1,\lambda_1,\alpha_1,\beta,0<\eta_1\leq 1, b$; $\gamma_0=1$,$\theta_0=1$. Set $k=1$. }
			\State{\textbf{Step 2}: Draw a mini-batch samples $I_{k}=\{\zeta_i^{k}\}_{i=1}^b$. Compute 
				\begin{align}
					v_{k}&=\nabla_x\tilde{G}_{k}( x_{k},y_{k};I_{k})+(1-\gamma_{k-1})[v_{k-1}-\nabla _x\tilde{G}_{k-1}(x_{k-1},y_{k-1};I_{k})]\label{vk}
				\end{align}			
				and 
				\begin{align}
					w_{k}&=\nabla_y\tilde{G}_{k}( x_{k},y_{k};I_{k})+(1-\theta_{k-1})[w_{k-1}-\nabla_y\tilde{G}_{k-1}( x_{k-1},y_{k-1};I_{k})]\label{wk}
				\end{align}
				where $\nabla_x\tilde{G}_{k}(x,y;I)$, $\nabla_x\tilde{G}_{k-1}(x,y;I)$ and $\nabla_y\tilde{G}_{k}(x,y;I)$, $\nabla_y\tilde{G}_{k-1}(x,y;I)$ are defined as in \eqref{sec2:Gkx} and \eqref{sec2:Gky}.}
			\State{\textbf{Step 3}:Perform the following update for $x_k$ and $y_k$:  	
				\qquad \begin{align}
					\tilde{x}_{k+1}&=\mathcal{P}_{\mathcal{X}}^{1/\alpha_k} \left( x_k - \alpha_kv_k\right),\label{sfo1:update-tx}\\ x_{k+1}&=x_k+\eta_k(\tilde{x}_{k+1}-x_k),\label{sfo1:update-x}\\
					\tilde{y}_{k+1}&=\mathcal{P}_{\mathcal{Y}}^{1/\beta} \left( y_k + \beta w_k\right),\label{sfo1:update-ty}\\ y_{k+1}&=y_k+\eta_k(\tilde{y}_{k+1}-y_k).\label{sfo1:update-y}
			\end{align}}
			\State{\textbf{Step 4}:If some stationary condition is satisfied, stop; otherwise, set $k=k+1, $ go to Step 2.}
		\end{algorithmic}
	\end{algorithm}

	Note that FORMDA algorithm differs from the Acc-MDA algorithm proposed in \cite{Huang} in two ways. On one hand, at each iteration of the FORMDA algorithm, the gradient of $\tilde{G}_{k}(x,y;\zeta)$, a regularized version of $G(x,y;\zeta)$, is computed and used, instead of the gradient of $G(x,y;\zeta)$ that used in the Acc-MDA algorithm. On the other hand, the FORMDA algorithm is designed to solve general nonconvex-concave minimax problems, whereas the Acc-MDA algorithm can only solve nonconvex-strongly concave minimax problems.

	Before we prove the iteration complexity of Algorithm \ref{sfoalg:1}, we first give some mild assumptions.
	\begin{ass}\label{azoass:Lip}
		For any given $\zeta$, $G(x,y,\zeta)$ has Lipschitz continuous gradients and there exist a constant $l>0$ such that for any $x, x_1, x_2\in \mathcal{X}$, and $y, y_1, y_2\in \mathcal{Y}$, we have
		\begin{align*}
			\| \nabla_{x} G(x_1,y,\zeta)-\nabla_{x} G(x_2,y,\zeta)\| &\leq l\|x_{1}-x_{2}\|,\\
			\|\nabla_{x} G(x,y_1,\zeta)-\nabla_{x} G(x,y_2,\zeta)\| &\leq l\|y_{1}-y_{2}\|,\\
			\|\nabla_{y}G(x,y_1,\zeta)-\nabla_{y} G(x,y_2,\zeta)\| &\leq l\|y_{1}-y_{2}\|,\\
			\|\nabla_{y} G(x_1,y,\zeta)-\nabla_{y} G(x_2,y,\zeta)\| &\leq l\|x_{1}-x_{2}\|.
		\end{align*}
	\end{ass}
	\begin{ass}\label{sazocrho}
		$\{\rho_k\}$ is a nonnegative monotonically decreasing sequence.
	\end{ass}
	
	If Assumption \ref{azoass:Lip} holds, $g(x,y)$ has Lipschitz continuous gradients with constant $l$ by Lemma 7 in \cite{xu22zeroth}. Then, by the definition of $g_k(x,y)$ and Assumption \ref{sazocrho}, we know that $g_k(x,y)$ has Lipschitz continuous gradients with constant $L$, where $L=l+\rho_1$.
	
	\begin{ass}\label{afoass:var}
		For any given $\zeta$, there exists a constant $\delta>0$ such that for all $x$ and $y$, it has
		\begin{align*}
			\mathbb{E}[\|\nabla_x  G(x,y,\zeta)-\nabla_xg(x,y) \|^2]&\le\delta^2,\\
			\mathbb{E}[\|\nabla_y  G(x,y,\zeta)-\nabla_yg(x,y) \|^2]&\le\delta^2.
		\end{align*}
	\end{ass}
	By Assumption \ref{afoass:var}, we can immediately get that 
	\begin{align}
		\mathbb{E}[\|\nabla_x  \tilde{G}_k(x,y;I)-\nabla_xg_k(x,y) \|^2]&\le\frac{\delta^2}{b},\label{ass2:Gx}\\
		\mathbb{E}[\|\nabla_y  \tilde{G}_k(x,y;I)-\nabla _yg_k(x,y) \|^2]&\le\frac{\delta^2}{b}.\label{ass2:Gy}
	\end{align}

	We define the stationarity gap as the termination criterion as follows.
	\begin{defi}\label{azogap-f}
		For some given $\alpha_k>0$ and $\beta>0$, the stationarity gap for problem \eqref{prob-s} is defined as 
		\begin{equation*}
			\nabla \mathcal{G}^{\alpha_k,\beta}( x,y) :=\left( \begin{array}{c}
				\frac{1}{\alpha_k}\left( x-\mathcal{P}_{\mathcal{X}}^{\frac{1}{\alpha_k}}\left( x-\alpha_k\nabla _xg( x,y) \right) \right)\\
				\frac{1}{\beta} \left( y-\mathcal{P}_{\mathcal{Y}}^{\frac{1}{\beta} }\left( y+ \beta\nabla _yg( x,y) \right) \right)
			\end{array} \right) .
		\end{equation*}
	\end{defi}

	\subsection{Complexity analysis.}
	In this subsection, we prove the iteration complexity of Algorithm \ref{sfoalg:1}. 
	When strong concavity of $g(x,y)$ with respect to $y$ is absent, $y^*(x)=\arg\max_{y\in\mathcal{Y}}g(x,y)$  is not necessarily unique, and $\Phi(x)=\max_{y\in\mathcal{Y}}g(x,y)$ is not necessarily smooth. To address these challenges, we analyze the iteration complexity of FORMDA algorithm by utilizing  $\Phi_k(x)$  and $y_k^*(x)$, where
	
	\begin{align}
		\Phi_k(x)&:=\max_{y\in\mathcal{Y}}g_k(x,y),\label{phik}\\
		y_k^*(x)&:=\arg\max_{y\in\mathcal{Y}}g_k(x,y).\label{yk*}
	\end{align}
	By Lemma 24 in \cite{nouiehed2019solving}, $\Phi_k(x)$ is $L_{\Phi_k}$-Lipschitz smooth with $L_{\Phi_k}=L+\frac{L^2}{\rho_k}$ under the $\rho_k$-strong concavity of $g_k(x,y)$. Moreover, similar to the proof of Lemma B.2 in \cite{lin2020near}, we have $\nabla_x\Phi_k(x)=\nabla_x g_k(x, y_k^*(x))$.

	We will begin to analyze the iteration complexity of the FORMDA algorithm, mainly through constructing potential functions to obtain the recursion. Firstly, we estimate an upper bound of $\|y_{k+1}^*(\bar{x})-y_{k}^*(x)\|^2$, which will be utilized in the subsequent construction of the potential function.
	
	\begin{lem}\label{sazoclem1}
		Suppose that Assumptions \ref{azoass:Lip} and \ref{sazocrho} hold. Then for any $x, \bar{x}\in\mathcal{X}$,
		\begin{align}\label{sazoclem1:1}
			\|y_{k+1}^*(\bar{x})-y_{k}^*(x)\|^2
			\le\frac{L^2}{\rho_{k+1}^2}\|\bar{x}-x\|^2+\frac{\rho_k-\rho_{k+1}}{\rho_{k+1}}(\|y_{k+1}^*(\bar{x})\|^2-\|y_{k}^*(x)\|^2).
		\end{align}
	\end{lem}
	
	\begin{proof}
		The optimality condition for $y_k^*(x)$ implies that $\forall y\in \mathcal{Y}$ and $\forall k\geq 1$,
		\begin{align}
			\langle\nabla_yg_{k+1}(\bar{x},y_{k+1}^*(\bar{x})), y-y_{k+1}^*(\bar{x})\rangle&\le0,\label{sazoclem1:2}\\
			\langle\nabla_yg_k(x,y_{k}^*(x)), y-y_{k}^*(x)\rangle&\le0.\label{sazoclem1:3}
		\end{align}
		Setting $y=y_{k}^*(x)$ in \eqref{sazoclem1:2} and $y=y_{k+1}^*(\bar{x})$ in \eqref{sazoclem1:3}, adding these two inequalities and using the strong concavity of $g_k(x,y)$ with respect to $y$, we have
		\begin{align}
			&\langle\nabla_yg_{k+1}(\bar{x},y_{k+1}^*(\bar{x}))-\nabla_yg_k(x,y_{k+1}^*(\bar{x})), y_{k}^*(x)-y_{k+1}^*(\bar{x})\rangle\nonumber\\
			\le&\langle\nabla_yg_k(x,y_{k+1}^*(\bar{x}))-\nabla_yg_k(x,y_{k}^*(x)), y_{k+1}^*(\bar{x})-y_{k}^*(x)\rangle\nonumber\\
			\le&-\rho_k\|y_{k+1}^*(\bar{x})-y_{k}^*(x)\|^2.\label{sazoclem1:4}
		\end{align}
		By the definition of $g_k(x,y)$, the Cauchy-Schwarz inequality and Assumption \ref{azoass:Lip}, \eqref{sazoclem1:4} implies that
		\begin{align}
			&(\rho_k-\rho_{k+1})\langle y_{k+1}^*(\bar{x}),y_{k}^*(x)-y_{k+1}^*(\bar{x}) \rangle\nonumber\\
			\le&\langle\nabla_yg(\bar{x},y_{k+1}^*(\bar{x}))-\nabla_yg(x,y_{k+1}^*(\bar{x})), y_{k+1}^*(\bar{x})-y_{k}^*(x)\rangle-\rho_k\|y_{k+1}^*(\bar{x})-y_{k}^*(x)\|^2\nonumber\\
			\le&\frac{L^2}{2\rho_k}\|\bar{x}-x\|^2-\frac{\rho_k}{2}\|y_{k+1}^*(\bar{x})-y_{k}^*(x)\|^2.\label{sazoclem1:5}
		\end{align}
		Since $\langle y_{k+1}^*(\bar{x}),y_{k}^*(x)-y_{k+1}^*(\bar{x}) \rangle=\frac{1}{2}(\|y_{k}^*(x)\|^2-\|y_{k+1}^*(\bar{x})\|^2-\|y_{k+1}^*(\bar{x})-y_{k}^*(x)\|^2)$ and Assumption \ref{sazocrho}, \eqref{sazoclem1:5} implies that
		\begin{align*}
			\|y_{k+1}^*(\bar{x})-y_{k}^*(x)\|^2
			\le\frac{L^2}{\rho_{k+1}^2}\|\bar{x}-x\|^2+\frac{\rho_k-\rho_{k+1}}{\rho_{k+1}}(\|y_{k+1}^*(\bar{x})\|^2-\|y_{k}^*(x)\|^2).
		\end{align*}
		The proof is completed.
	\end{proof}

	Next, we provide an estimate of the difference between $\Phi_{k+1}(x_{k+1})$ and $\Phi_k(x_{k})$.
	
	\begin{lem}\label{safoclem2}
		Suppose that Assumptions \ref{azoass:Lip} and \ref{sazocrho} hold. Let $\{(x_k,y_k)\}$ be a sequence generated by Algorithm \ref{sfoalg:1}, if $\rho_k\le L$, then $\forall k \ge 1$, 
		\begin{align}
			\Phi_{k+1}(x_{k+1})-\Phi_k(x_{k})
			\le&2\eta_k\alpha_k L^2\|y_k-y_k^*(x_k)\|^2+2\eta_k\alpha_k\|\nabla _xg_k(x_{k},y_{k})-v_k\|^2\nonumber\\
			&-(\frac{3\eta_k}{4\alpha_k}-\frac{L^2\eta_k^2}{\rho_{k+1}})\|\tilde{x}_{k+1}-x_k\|^2+\frac{\rho_k-\rho_{k+1}}{2}\sigma_y^2,\label{safoclem2:1}
		\end{align}
		where $\sigma_y=\max\{\|y\| \mid y\in\mathcal{Y}\}$.
	\end{lem}
	
	\begin{proof}
		Since that $\Phi_k(x)$ is $L_{\Phi_k}$-smooth with respect to $x$ and $\rho_k\le L$, we have that
		\begin{align} 
			&\Phi_k(x_{k+1})-\Phi_k(x_{k})\nonumber\\
			\le &\langle \nabla _x\Phi_k(x_{k}) ,x_{k+1}-x_k \rangle  +\frac{L_{\Phi_k}}{2}\|x_{k+1}-x_k\|^2\nonumber\\
			\le&\eta_k\langle \nabla _xg_k(x_{k},y_k^*(x_k))-\nabla_xg_k(x_{k},y_{k}) ,\tilde{x}_{k+1}-x_k \rangle+\eta_k\langle v_k ,\tilde{x}_{k+1}-x_k \rangle  \nonumber\\
			&+\eta_k\langle \nabla_xg_k(x_{k},y_{k})-v_k,\tilde{x}_{k+1}-x_k \rangle+\frac{L^2\eta_k^2}{\rho_k}\| \tilde{x}_{k+1}-x_k\|^2.\label{safoclem2:2}
		\end{align}
		Next, we estimate the first three terms in the right hand side of \eqref{safoclem2:2}. By the Cauchy-Schwarz inequality, we get
		\begin{align}
			&\langle \nabla _xg_k(x_{k},y_k^*(x_k))-\nabla _xg_k(x_{k},y_{k}),\tilde{x}_{k+1}-x_k \rangle\nonumber\\
			\le&2\alpha_k\|\nabla _xg_k(x_{k},y_k^*(x_k))-\nabla _xg_k(x_{k},y_{k})\|^2+\frac{1}{8\alpha_k}\|\tilde{x}_{k+1}-x_k\|^2\nonumber\\
			\le&2\alpha_k L^2\|y_k-y_k^*(x_k)\|^2+\frac{1}{8\alpha_k}\|\tilde{x}_{k+1}-x_k\|^2.\label{safoclem2:3}
		\end{align}
		By the Cauchy-Schwarz inequality, we have
		\begin{align}
			&\langle \nabla _xg_k(x_{k},y_{k})-v_k,\tilde{x}_{k+1}-x_k \rangle\nonumber\\
			\le&2\alpha_k\|\nabla _xg_k(x_{k},y_{k})-v_k\|^2+\frac{1}{8\alpha_k}\|\tilde{x}_{k+1}-x_k\|^2.\label{safoclem2:4}
		\end{align}
		The optimality condition for $x_k$ in \eqref{sfo1:update-tx} implies that $\forall x\in \mathcal{X}$ and $\forall k\geq 1$,
		\begin{equation}\label{safoclem2:5}
			\langle v_k,\tilde{x}_{k+1}-x_k \rangle \le -\frac{1}{\alpha_k}\|\tilde{x}_{k+1}-x_k\|^2.
		\end{equation}
		Plugging \eqref{safoclem2:3}, \eqref{safoclem2:4} and \eqref{safoclem2:5} into \eqref{safoclem2:2} and using $\rho_{k+1}\le\rho_k$, we get
		\begin{align}
			\Phi_k(x_{k+1})-\Phi_k(x_{k})
			\le&2\eta_k\alpha_k L^2\|y_k-y_k^*(x_k)\|^2+2\eta_k\alpha_k\|\nabla _xg_k(x_{k},y_{k})-v_k\|^2\nonumber\\
			&-(\frac{3\eta_k}{4\alpha_k}-\frac{L^2\eta_k^2}{\rho_{k+1}})\|\tilde{x}_{k+1}-x_k\|^2.\label{safoclem2:6}
		\end{align}
		On the other hand, we have
		\begin{align}
			\Phi_{k+1}(x_{k+1})-\Phi_{k}(x_{k+1})
			\le&\Phi_{k+1}(x_{k+1})-g_k(x_{k+1},y_{k+1}^*(x_{k+1}))\nonumber\\
			=&\frac{\rho_k-\rho_{k+1}}{2}\|y_{k+1}^*(x_{k+1})\|^2\le\frac{\rho_k-\rho_{k+1}}{2}\sigma_y^2.\label{safoclem2:7}
		\end{align}
		Combining \eqref{safoclem2:6} and \eqref{safoclem2:7}, we complete the proof.
	\end{proof}

	We further to estimate the upper bound of the first two terms in the right-hand side of \eqref{safoclem2:1} in Lemma \ref{safoclem2}. 
	
	\begin{lem}\label{safoclem3}
		Suppose that Assumptions \ref{azoass:Lip} and \ref{sazocrho} hold. Let $\{(x_k,y_k)\}$ be a sequence generated by Algorithm \ref{sfoalg:1}, if $0<\eta_k\le1$, $0<\beta\le\frac{1}{6L}$ and $\rho_k\le L$, then $\forall k \ge 1$, 
		\begin{align}
			&\|y_{k+1}-y_{k+1}^*(x_{k+1})\|^2\nonumber\\
			\le&(1-\frac{\eta_k\beta\rho_{k+1}}{4})\|y_{k}-y_k^*(x_k)\|^2-\frac{3\eta_k}{4}\|\tilde{y}_{k+1} -y_k\|^2+\frac{5L^2\eta_k}{\rho_{k+1}^3\beta}\|\tilde{x}_{k+1}-x_k\|^2\nonumber\\
			&+\frac{5\eta_k\beta}{\rho_{k+1}}\|\nabla_{y}g_k(x_{k},y_{k})-w_k\|^2+\frac{5(\rho_k-\rho_{k+1})}{\rho_{k+1}^2\eta_k\beta}(\|y_{k+1}^*(x_{k+1})\|^2-\|y_{k}^*(x_k)\|^2).\label{safoclem3:1}
		\end{align}
	\end{lem}
	
	\begin{proof}
		$g_k(x,y)$ is $\rho_k$-strongly concave with respect to $y$, which implies that
		\begin{align}\label{safoclem3:2}
			&g_k(x_{k},y)-g_k( x_{k},y_{k})\nonumber\\
			\le& \langle \nabla _{y}g_k(x_{k},y_{k}),y-y_{k} \rangle -\frac{\rho_k}{2}\| y-y_{k} \|^2\nonumber\\
			=& \langle w_k,y-\tilde{y}_{k+1} \rangle +\langle \nabla _{y}g_k(x_{k},y_{k})-w_k,y-\tilde{y}_{k+1} \rangle +\langle \nabla _{y}g_k(x_{k},y_{k}),\tilde{y}_{k+1}-y_{k} \rangle\nonumber\\
			&-\frac{\rho_k}{2}\| y-y_{k} \|^2.
		\end{align}
		By  Assumption \ref{azoass:Lip}, $g_k(x,y)$ has Lipschitz continuous gradient with respect to $y$ and then
		\begin{align}\label{safoclem3:3}
			&g_k(x_{k},\tilde{y}_{k+1})-g_k( x_{k},y_{k})
			\ge\langle \nabla _{y}g_k(x_{k},y_{k}),\tilde{y}_{k+1}-y_{k} \rangle -\frac{L}{2}\| \tilde{y}_{k+1}-y_{k} \|^2.
		\end{align}
		The optimality condition for $y_k$ in \eqref{sfo1:update-ty} implies that $\forall y\in \mathcal{Y}$ and $\forall k\geq 1$,
		\begin{align}\label{safoclem3:4}
			\langle w_k,y-\tilde{y}_{k+1} \rangle &\le \frac{1}{\beta}\langle\tilde{y}_{k+1}-y_k,y-\tilde{y}_{k+1} \rangle\nonumber\\
			&=-\frac{1}{\beta}\|\tilde{y}_{k+1}-y_k\|^2+\frac{1}{\beta}\langle\tilde{y}_{k+1}-y_k,y-y_k \rangle.
		\end{align}
		Plugging \eqref{safoclem3:4} into \eqref{safoclem3:2} and combining \eqref{safoclem3:3}, by setting $y=y_k^*(x_k)$, we have
		\begin{align}\label{safoclem3:5}
			&g_k(x_{k},y_k^*(x_k))-g_k( x_{k},\tilde{y}_{k+1})\nonumber\\
			\le& \frac{1}{\beta}\langle\tilde{y}_{k+1}-y_k,y_k^*(x_k)-y_k \rangle+\langle \nabla _{y}g_k(x_{k},y_{k})-w_k,y_k^*(x_k)-\tilde{y}_{k+1} \rangle \nonumber\\
			&-\frac{\rho_k}{2}\|y_{k}-y_k^*(x_k) \|^2-(\frac{1}{\beta}-\frac{L}{2})\|\tilde{y}_{k+1}-y_k\|^2.
		\end{align}
		Next, we estimate the first two terms in the right hand side of \eqref{safoclem3:5}. By \eqref{sfo1:update-y}, we get
		\begin{align}
			&\|y_{k+1}-y_k^*(x_k)\|^2\nonumber\\
			=&\|y_{k}+\eta_k(\tilde{y}_{k+1}-y_k)-y_k^*(x_k)\|^2\nonumber\\
			=&\|y_k-y^*(x_k)\|^2+2\eta_k\langle\tilde{y}_{k+1}-y_k,y_k-y_k^*(x_k) \rangle+\eta_k^2\|\tilde{y}_{k+1}-y_k\|^2.\label{safoclem3:6}
		\end{align}
		\eqref{safoclem3:6} can be rewritten as 
		\begin{align}
			&\langle\tilde{y}_{k+1}-y_k,y_k^*(x_k)-y_k \rangle\nonumber\\
			=&\frac{1}{2\eta_k}\|y_k-y_k^*(x_k)\|^2+\frac{\eta_k}{2}\|\tilde{y}_{k+1}-y_k\|^2-\frac{1}{2\eta_k}	\|y_{k+1}-y_k^*(x_k)\|^2.\label{safoclem3:7}
		\end{align}
		By the Cauchy-Schwarz inequality, we have
		\begin{align}
			&\langle \nabla _{y}g_k(x_{k},y_{k})-w_k,y_k^*(x_k)-\tilde{y}_{k+1} \rangle \nonumber\\
			\le&\frac{2}{\rho_k}\|\nabla _{y}g_k(x_{k},y_{k})-w_k\|^2+\frac{\rho_k}{8}\|y^*(x_k)-\tilde{y}_{k+1} \|^2\nonumber\\
			\le&\frac{2}{\rho_k}\|\nabla _{y}g_k(x_{k},y_{k})-w_k\|^2+\frac{\rho_k}{4}\|y_k^*(x_k)-y_k\|^2+\frac{\rho_k}{4}\|\tilde{y}_{k+1} -y_k\|^2.\label{safoclem3:8}
		\end{align}
		Plugging \eqref{safoclem3:7}, \eqref{safoclem3:8} into \eqref{safoclem3:5}, and using the fact that $g_k(x_{k},y_k^*(x_k))\ge g_k( x_{k},\tilde{y}_{k+1})$, we get
		\begin{align}
			&\frac{1}{2\eta_k\beta}\|y_{k+1}-y_k^*(x_k)\|^2\nonumber\\
			\le&(\frac{1}{2\eta_k\beta}-\frac{\rho_k}{4})\|y_{k}-y_k^*(x_k)\|^2+(\frac{\eta_k}{2\beta}+\frac{\rho_k}{4}+\frac{L}{2}-\frac{1}{\beta})\|\tilde{y}_{k+1} -y_k\|^2\nonumber\\
			&+\frac{2}{\rho_k}\|\nabla _{y}g_k(x_{k},y_{k})-w_k\|^2.\label{safoclem3:9}
		\end{align}
		By the assumption $0<\eta_k\le1$, $0<\tilde{\beta}\le\frac{1}{6L}$, $\rho_k\le L$ and \eqref{safoclem3:9}, we obtain that
		\begin{align}
			&\|y_{k+1}-y_k^*(x_k)\|^2\nonumber\\
			\le&(1-\frac{\eta_k\beta\rho_k}{2})\|y_{k}-y_k^*(x_k)\|^2-\frac{3\eta_k}{4}\|\tilde{y}_{k+1} -y_k\|^2+\frac{4\eta_k\beta}{\rho_k}\|\nabla _{y}g_k(x_{k},y_{k})-w_k\|^2.\label{safoclem3:10}
		\end{align}
		By the Cauchy-Schwarz inequality, Lemma \ref{sazoclem1} and \eqref{sfo1:update-x}, we have
		\begin{align}
			&\|y_{k+1}-y_{k+1}^*(x_{k+1})\|^2\nonumber\\
			=&\|y_{k+1}-y_k^*(x_{k})\|^2+2\langle y_{k+1}-y_k^*(x_{k}),y_k^*(x_{k})-y_{k+1}^*(x_{k+1}) \rangle+\|y_k^*(x_{k})-y_{k+1}^*(x_{k+1})\|^2\nonumber\\
			\le&(1+\frac{\eta_k\beta\rho_k}{4})\|y_{k+1}-y_k^*(x_{k})\|^2+(1+\frac{4}{\eta_k\beta\rho_k})\|y_k^*(x_{k})-y_{k+1}^*(x_{k+1})\|^2\nonumber\\
			\le&(1+\frac{\eta_k\beta\rho_k}{4})\|y_{k+1}-y_k^*(x_{k})\|^2+(1+\frac{4}{\eta_k\beta\rho_k})\frac{\eta_k^2L^2}{\rho_{k+1}^2}\|\tilde{x}_{k+1}-x_k\|^2\nonumber\\
			&+(1+\frac{4}{\eta_k\beta\rho_k})\frac{\rho_k-\rho_{k+1}}{\rho_{k+1}}(\|y_{k+1}^*(x_{k+1})\|^2-\|y_{k}^*(x_k)\|^2).\label{safoclem3:11}
		\end{align}
		Plugging \eqref{safoclem3:10} into \eqref{safoclem3:11}, we obtain
		\begin{align}
			&\|y_{k+1}-y_{k+1}^*(x_{k+1})\|^2\nonumber\\
			\le&(1-\frac{\eta_k\beta\rho_k}{2})(1+\frac{\eta_k\beta\rho_k}{4})\|y_{k}-y_k^*(x_k)\|^2-\frac{3\eta_k}{4}(1+\frac{\eta_k\beta\rho_k}{4})\|\tilde{y}_{k+1} -y_k\|^2\nonumber\\
			&+(1+\frac{4}{\eta_k\beta\rho_k})\frac{\eta_k^2L^2}{\rho_{k+1}^2}\|\tilde{x}_{k+1}-x_k\|^2+\frac{4\eta_k\beta}{\rho_k}(1+\frac{\eta_k\beta\rho_k}{4})\|\nabla _{y}g_k(x_{k},y_{k})-w_k\|^2\nonumber\\
			&+(1+\frac{4}{\eta_k\beta\rho_k})\frac{\rho_k-\rho_{k+1}}{\rho_{k+1}}(\|y_{k+1}^*(x_{k+1})\|^2-\|y_{k}^*(x_k)\|^2).\label{safoclem3:12}
		\end{align}
		Since $0<\eta_k\le1$, $0<\beta\le\frac{1}{6L}$ and $\rho_k\le L$, we have $\eta_k\beta\rho_k<1$.Then, by Assumption \ref{sazocrho}, we get $(1-\frac{\eta_k\beta\rho_k}{2})(1+\frac{\eta_k\beta\rho_k}{4})\le1-\frac{\eta_k\beta\rho_{k+1}}{4}$, $-\frac{3\eta_k}{4}(1+\frac{\eta_k\beta\rho_k}{4})\le-\frac{3\eta_k}{4}$, $\frac{4\eta_k\beta}{\rho_k}(1+\frac{\eta_k\beta\rho_k}{4})\leq\frac{5\eta_k\beta}{\rho_{k+1}}$, $(1+\frac{4}{\eta_k\beta\rho_k})\le \frac{5}{\rho_{k+1}\beta\eta_k}$.
		Thus, we obtain
		\begin{align*}
			&\|y_{k+1}-y_{k+1}^*(x_{k+1})\|^2\nonumber\\
			\le&(1-\frac{\eta_k\beta\rho_{k+1}}{4})\|y_{k}-y_k^*(x_k)\|^2-\frac{3\eta_k}{4}\|\tilde{y}_{k+1} -y_k\|^2+\frac{5L^2\eta_k}{\rho_{k+1}^3\beta}\|\tilde{x}_{k+1}-x_k\|^2\nonumber\\
			&+\frac{5\eta_k\beta}{\rho_{k+1}}\|\nabla _{y}g_k(x_{k},y_{k})-w_k\|^2+\frac{5(\rho_k-\rho_{k+1})}{\rho_{k+1}^2\eta_k\beta}(\|y_{k+1}^*(x_{k+1})\|^2-\|y_{k}^*(x_k)\|^2).
		\end{align*}
		The proof is completed.
	\end{proof}

	Next, we provide upper bound estimates of $\mathbb{E}[\|\nabla _{x}g_{k+1}(x_{k+1},y_{k+1})-v_{k+1}\|^2]$ and $\mathbb{E}[\|\nabla_{y}g_{k+1}(x_{k+1},y_{k+1})-w_{k+1}\|^2]$ in the following lemma.
	
	\begin{lem}\label{safoclem4}
		Suppose that Assumptions \ref{azoass:Lip} and \ref{afoass:var} hold. Let $\{\left(x_k,y_k\right)\}$ be a sequence generated by Algorithm \ref{sfoalg:1}, then $\forall k \ge 1$, 
		\begin{align}
			&\mathbb{E}[\|\nabla _{x}g_{k+1}(x_{k+1},y_{k+1})-v_{k+1}\|^2]\nonumber\\
			\le&(1-\gamma_{k})\mathbb{E}[\|\nabla _{x}g_k(x_{k},y_{k})-v_k\|^2]+\frac{2\gamma_k^2\delta^2}{b}+\frac{2L^2\eta_k^2}{b}\mathbb{E}[\|\tilde{x}_{k+1}-x_k\|^2+\|\tilde{y}_{k+1}-y_k\|^2],\label{safoclem4:1x}\\
			&\mathbb{E}[\|\nabla_{y}g_{k+1}(x_{k+1},y_{k+1})-w_{k+1}\|^2]\nonumber\\
			\le&(1-\theta_{k})\mathbb{E}[\|\nabla _{y}g_k(x_{k},y_{k})-w_k\|^2]+\frac{2\theta_k^2\delta^2}{b}+\frac{4L^2\eta_k^2}{b}\mathbb{E}[\|\tilde{x}_{k+1}-x_k\|^2+\|\tilde{y}_{k+1}-y_k\|^2]\nonumber\\
			&+\frac{4(\rho_k^2-\rho_{k+1}^2)\sigma_y^2}{b}.\label{safoclem4:1y}
		\end{align}
	\end{lem}
	
	\begin{proof}
		Note that $\mathbb{E}[\nabla_{x}\tilde{G}_k(x_{k},y_{k};I_{k+1})]=\nabla _{x}g_k(x_{k},y_{k})$, $\mathbb{E}[\nabla _{x}\tilde{G}_{k+1}(x_{k+1},y_{k+1};I_{k+1})]=\nabla_{x}g_{k+1}(x_{k+1},y_{k+1})$, and by the definition of $v_{k+1}$, we have
		\begin{align}
			&\mathbb{E}[\|\nabla _{x}g_{k+1}(x_{k+1},y_{k+1})-v_{k+1}\|^2]\nonumber\\
			=&\mathbb{E}[\|\nabla _{x}g_{k+1}(x_{k+1},y_{k+1})-\nabla_x\tilde{G}_{k+1}( x_{k+1},y_{k+1};I_{k+1})\nonumber\\
			&-(1-\gamma_{k})[v_k-\nabla_x\tilde{G}_{k}( x_{k},y_{k};I_{k+1})]\|^2]\nonumber\\
			=&\mathbb{E}[\|(1-\gamma_{k})(\nabla _{x}g_k(x_{k},y_{k})-v_k)+\gamma_k(\nabla _{x}g_{k+1}(x_{k+1},y_{k+1})\nonumber\\
			&-\nabla_x\tilde{G}_{k+1}( x_{k+1},y_{k+1};I_{k+1}))+(1-\gamma_{k})[\nabla_xg_{k+1}( x_{k+1},y_{k+1})-\nabla_xg_k( x_{k},y_{k})\nonumber\\
			&-\nabla_x\tilde{G}_{k+1}( x_{k+1},y_{k+1};I_{k+1})+\nabla_x\tilde{G}_k( x_{k},y_{k};I_{k+1})]\|^2]\nonumber\\
			=&(1-\gamma_{k})^2\mathbb{E}[\|\nabla _{x}g_k(x_{k},y_{k})-v_k\|^2]+\mathbb{E}[\|\gamma_k(\nabla _{x}g_{k+1}(x_{k+1},y_{k+1})\nonumber\\
			&-\nabla_x\tilde{G}_{k+1}(x_{k+1},y_{k+1};I_{k+1}))+(1-\gamma_{k})[\nabla _xg_{k+1}( x_{k+1},y_{k+1})-\nabla_xg_k( x_{k},y_{k})\nonumber\\
			&-\nabla_x\tilde{G}_{k+1}( x_{k+1},y_{k+1};I_{k+1})+\nabla_x\tilde{G}_{k}( x_{k},y_{k};I_{k+1})]\|^2].\label{safoclem4:2}
		\end{align}
		By the fact that $\mathbb{E}[\|\zeta-\mathbb{E}\zeta\|^2]=\mathbb{E}[\|\zeta\|^2]-\|\mathbb{E}[\zeta\|^2]\le\mathbb{E}[\|\zeta\|^2]$,  $\mathbb{E}[\|\frac{1}{b}\sum_{j=1}^{b}\zeta_j\|^2]=\frac{1}{b}\mathbb{E}[\|\zeta_j\|^2]$ for i.i.d. random variables $\{\zeta_j\}_{j=1}^b$ with zero mean, $1-\gamma_{k}<1$ and \eqref{ass2:Gx}, we have
		\begin{align}
			&\mathbb{E}[\|\nabla_{x}g_{k+1}(x_{k+1},y_{k+1})-v_{k+1}\|^2]\nonumber\\
			\le&(1-\gamma_{k})^2\mathbb{E}[\|\nabla _{x}g_k(x_{k},y_{k})-v_k\|^2]+\frac{2\gamma_k^2\delta^2}{b}\nonumber\\
			&+\frac{2(1-\gamma_k)^2}{b}\mathbb{E}[\|\nabla_x\tilde{G}_{k+1}( x_{k+1},y_{k+1};\zeta_1^{k+1})-\nabla_x\tilde{G}_{k}( x_{k},y_{k};\zeta_1^{k+1})\|^2]\nonumber\\
			\le&(1-\gamma_{k})\mathbb{E}[\|\nabla _{x}g_k(x_{k},y_{k})-v_k\|^2]+\frac{2\gamma_k^2\delta^2}{b}+\frac{2L^2\eta_k^2}{b}\mathbb{E}[\|\tilde{x}_{k+1}-x_k\|^2+\|\tilde{y}_{k+1}-y_k\|^2].\label{safoclem4:3}
		\end{align}
		Similarly, we get
		\begin{align*}
			&\mathbb{E}[\|\nabla_{y}g_{k+1}(x_{k+1},y_{k+1})-w_{k+1}\|^2]\nonumber\\
			\le&(1-\theta_{k})^2\mathbb{E}[\|\nabla _{y}g_k(x_{k},y_{k})-w_k\|^2]+\frac{2\theta_k^2\delta^2}{b}\nonumber\\
			&+\frac{2(1-\theta_k)^2}{b}\mathbb{E}[\|\nabla_y\tilde{G}_{k+1}( x_{k+1},y_{k+1};\zeta_1^{k+1})-\nabla_y\tilde{G}_{k}( x_{k},y_{k};\zeta_1^{k+1})\|^2]\nonumber\\
			\le&(1-\theta_{k})\mathbb{E}[\|\nabla _{y}g_k(x_{k},y_{k})-w_k\|^2]+\frac{2\theta_k^2\delta^2}{b}+\frac{4(\rho_k^2-\rho_{k+1}^2)\sigma_y^2}{b}\nonumber\\
			&+\frac{4L^2\eta_k^2}{b}\mathbb{E}[\|\tilde{x}_{k+1}-x_k\|^2+\|\tilde{y}_{k+1}-y_k\|^2].
		\end{align*}
		The proof is completed.
	\end{proof}

	We now establish an important recursion for the FORMDA algorithm.
	\begin{lem}\label{safoclem45}
		Suppose that Assumptions  \ref{azoass:Lip}, \ref{sazocrho} and \ref{afoass:var} hold. Let $\{\left(x_k,y_k\right)\}$ be a sequence generated by Algorithm \ref{sfoalg:1}. Denote
		\begin{align*}
			&F_{k+1}(x_{k+1},y_{k+1})\\
			=&\mathbb{E}[\Phi_{k+1}(x_{k+1})]+D_k^{(1)}\mathbb{E}[\|\nabla_{x}g_{k+1}(x_{k+1},y_{k+1})-v_{k+1}\|^2]\\
			&+\frac{8\alpha_kL^2}{\beta\rho_{k+1}}\mathbb{E}[\|y_{k+1}-y_{k+1}^*(x_{k+1})\|^2]+D_k^{(2)}\mathbb{E}[\|\nabla _{y}g_{k+1}(x_{k+1},y_{k+1})-w_{k+1}\|^2],\\
			&S_{k+1}(x_{k+1},y_{k+1})\\
			=&F_{k+1}(x_{k+1},y_{k+1})-\frac{40\alpha_{k+1}L^2(\rho_{k+1}-\rho_{k+2})}{\eta_{k+1}\beta^2\rho_{k+2}^3}\mathbb{E}[\|y_{k+1}^*(x_{k+1})\|^2]+\frac{4D_{k+1}^{(2)}\rho_{k+1}^2\sigma_y^2}{b}.
		\end{align*}
		Then $\forall k \ge 1$, 
		\begin{align}\label{safoclem45:1}
			&S_{k+1}(x_{k+1},y_{k+1})-S_{k}(x_{k},y_{k})\nonumber\\
			\le&(\frac{8\alpha_kL^2}{\beta\rho_{k+1}}-\frac{8\alpha_{k-1}L^2}{\beta\rho_{k}})\mathbb{E}[\|y_k-y_k^*(x_k)\|^2]\nonumber\\
			&+(2\eta_k\alpha_k-D_k^{(1)}\gamma_{k}+D_k^{(1)}-D_{k-1}^{(1)})\mathbb{E}[\|\nabla _{x}g_k(x_{k},y_{k})-v_k\|^2]\nonumber\\
			&+(\frac{40\eta_k\alpha_kL^2}{\rho_{k+1}^2}-D_k^{(2)}\theta_{k}+D_k^{(2)}-D_{k-1}^{(2)})\mathbb{E}[\|\nabla _{y}g_k(x_{k},y_{k})-w_k\|^2]\nonumber\\
			&-(\frac{3\eta_k}{4\alpha_k}-\frac{L^2\eta_k^2}{\rho_{k+1}}-\frac{40L^4\eta_k\alpha_k}{\rho_{k+1}^4\beta^2}-\frac{4L^2\eta_k^2(D_k^{(1)}d_x+D_k^{(2)}d_y)}{b})\mathbb{E}[\|\tilde{x}_{k+1}-x_k\|^2]\nonumber\\
			&-(\frac{6\alpha_kL^2\eta_k}{\rho_{k+1}\beta}-\frac{4L^2\eta_k^2(D_k^{(1)}d_x+D_k^{(2)}d_y)}{b})\mathbb{E}[\|\tilde{y}_{k+1} -y_k\|^2]\nonumber\\
			&+(\frac{40\alpha_kL^2(\rho_k-\rho_{k+1})}{\eta_{k}\beta^2\rho_{k+1}^3}-\frac{40\alpha_{k+1}L^2(\rho_{k+1}-\rho_{k+2})}{\eta_{k+1}\beta^2\rho_{k+1}^3})\sigma_y^2+\frac{\rho_k-\rho_{k+1}}{2}\sigma_y^2\nonumber\\
			&+\frac{2\delta^2(\gamma_k^2D_k^{(1)}+\theta_k^2D_k^{(2)})}{b}+\frac{4\rho_{k+1}^2(D_{k+1}^{(2)}-D_k^{(2)})\sigma_y^2}{b},
		\end{align}
		where $\sigma_y=\max\{\|y\| \mid y\in\mathcal{Y}\}$, $D_k^{(1)}>0$ and $D_k^{(2)}>0$.
	\end{lem}
	
	\begin{proof}
		By the definition of $F_{k}(x_{k},y_{k})$ and Lemmas \ref{safoclem2}-\ref{safoclem4}, we get
		\begin{align}\label{safocthm1:2}
			&F_{k+1}(x_{k+1},y_{k+1})-F_{k}(x_{k},y_{k})\nonumber\\
			\le&(\frac{8\alpha_kL^2}{\beta\rho_{k+1}}-\frac{8\alpha_{k-1}L^2}{\beta\rho_{k}})\mathbb{E}[\|y_k-y_k^*(x_k)\|^2]\nonumber\\
			&+(2\eta_k\alpha_k-D_k^{(1)}\gamma_{k}+D_k^{(1)}-D_{k-1}^{(1)})\mathbb{E}[\|\nabla _{x}g_k(x_{k},y_{k})-v_k\|^2]\nonumber\\
			&+(\frac{40\eta_k\alpha_kL^2}{\rho_{k+1}^2}-D_k^{(2)}\theta_{k}+D_k^{(2)}-D_{k-1}^{(2)})\mathbb{E}[\|\nabla _{y}g_k(x_{k},y_{k})-w_k\|^2]\nonumber\\
			&-(\frac{3\eta_k}{4\alpha_k}-\frac{L^2\eta_k^2}{\rho_{k+1}}-\frac{40L^4\eta_k\alpha_k}{\rho_{k+1}^4\beta^2}-\frac{4L^2\eta_k^2(D_k^{(1)}d_x+D_k^{(2)}d_y)}{b})\mathbb{E}[\|\tilde{x}_{k+1}-x_k\|^2]\nonumber\\
			&-(\frac{6\alpha_kL^2\eta_k}{\rho_{k+1}\beta}-\frac{4L^2\eta_k^2(D_k^{(1)}d_x+D_k^{(2)}d_y)}{b})\mathbb{E}[\|\tilde{y}_{k+1} -y_k\|^2]\nonumber\\
			&+\frac{40\alpha_kL^2(\rho_k-\rho_{k+1})}{\eta_{k}\beta^2\rho_{k+1}^3}(\mathbb{E}[\|y_{k+1}^*(x_{k+1})\|^2]-\mathbb{E}[\|y_{k}^*(x_k)\|^2])+\frac{\rho_k-\rho_{k+1}}{2}\sigma_y^2\nonumber\\
			&+\frac{2\delta^2(\gamma_k^2D_k^{(1)}+\theta_k^2D_k^{(2)})}{b}+\frac{4D_k^{(2)}(\rho_k^2-\rho_{k+1}^2)\sigma_y^2}{b}.
		\end{align}
		The proof is then completed by the definition of $S_{k}(x_{k},y_{k})$ and $\sigma_y$.
	\end{proof}

	Define $T(\varepsilon):=\min\{k \mid \|\nabla \mathcal{G}^{\alpha_k,\beta }(x_k,y_k)\|\leq \varepsilon \}$, where $\varepsilon>0$ is a given target accuracy. Denote
	\begin{equation*}
		\nabla \tilde{\mathcal{G}}_k^{\alpha_k,\beta}( x,y) =\left( \begin{array}{c}
			\frac{1}{\alpha_k}\left( x-\mathcal{P}_{\mathcal{X}}^{\frac{1}{\alpha_k}}\left( x-\alpha_k\nabla _xg_k( x,y) \right) \right)\\
			\frac{1}{\beta} \left( y-\mathcal{P}_{\mathcal{Y}}^{\frac{1}{\beta} }\left( y+ \beta\nabla _yg_k( x,y) \right) \right)
		\end{array} \right) .
	\end{equation*}
	It can be easily checked that
	$$\|\nabla \mathcal{G}^{\alpha_k,\beta}(x,y)\|\le\|\nabla \tilde{\mathcal{G}}_k^{\alpha_k,\beta}(x,y)\|+\rho_{k}\|y\|,$$
	and $\|\nabla \mathcal{G}^{\alpha_k,\beta}(x,y)\|=\|\nabla \tilde{\mathcal{G}}_k^{\alpha_k,\beta}(x,y)\|$ if $\rho_k=0$.
	Next, we provide an upper bound estimate of $\|\nabla \tilde{\mathcal{G}}_k^{\alpha_k,\beta}(x,y)\|$.
	
	\begin{lem}\label{safoclem5}
		Suppose that Assumption \ref{azoass:Lip} holds. Let $\{\left(x_k,y_k\right)\}$ be a sequence generated by Algorithm \ref{sfoalg:1}, then $\forall k \ge 1$, 
		\begin{align}\label{safoclem5:1}
			&\mathbb{E}[\|\nabla \tilde{\mathcal{G}}_k^{\alpha_k,\beta}(x_k,y_k)\|^2]\nonumber\\
			\le&\frac{2}{\alpha_k^2} \mathbb{E}[\|\tilde{x}_{k+1}-x_k\|^2]+\frac{2}{\beta^2} \mathbb{E}[\|\tilde{y}_{k+1}-y_k\|^2]+2\mathbb{E}[\|\nabla_xg_k\left( x_k,y_{k}\right)-v_k\|^2]\nonumber\\
			&+2\mathbb{E}[\|\nabla_yg_k\left( x_k,y_{k} \right)-w_k\|^2].
		\end{align}
	\end{lem}
	
	\begin{proof}
		By \eqref{sfo1:update-ty}, the nonexpansive property of the projection operator, we immediately get
		\begin{align}
			\| \frac{1}{\beta} ( y_k-\mathcal{P}_{\mathcal{Y}}^{\frac{1}{\beta}}\left( y_k+ \beta\nabla _yg_k\left( x_k,y_k\right) \right) ) \|
			\le\frac{1}{\beta} \|\tilde{y}_{k+1}-y_k\|+\|\nabla _yg_k\left( x_k,y_{k} \right)-w_k\|.\label{safoclem5:2}
		\end{align}
		On the other hand, by \eqref{sfo1:update-tx} and the nonexpansive property of the projection operator, we have
		\begin{align}
			\| \frac{1}{\alpha_k} ( x_k-\mathcal{P}_{\mathcal{X}}^{\frac{1}{\alpha_k}}\left( x_k- \alpha_k\nabla _xg_k\left( x_k,y_k\right) \right)) \|
			\le\frac{1}{\alpha_k} \|\tilde{x}_{k+1}-x_k\|+\|\nabla _xg_k\left( x_k,y_{k} \right)-v_k\|.\label{safoclem5:3}
		\end{align}
		Combing \eqref{safoclem5:2}, \eqref{safoclem5:3}, using Cauchy-Schwarz inequality and taking the expectation, we complete the proof.
	\end{proof}
	
	\begin{thm}\label{safocthm1}
		Suppose that Assumptions \ref{azoass:Lip}, \ref{sazocrho} and \ref{afoass:var} hold. Let $\{\left(x_k,y_k\right)\}$ be a sequence generated by Algorithm \ref{sfoalg:1}. 
		For any given $k\geq 1$, let
		\begin{align*}
			\eta_k&=\frac{1}{(k+2)^{5/13}},\quad \alpha_k=\frac{a_4}{(k+2)^{4/13}},\quad \rho_{k}=\frac{L}{(k+1)^{2/13}},\\
			\gamma_k&=\frac{a_5}{(k+2)^{12/13}},\quad \theta_k=\frac{a_6}{(k+2)^{8/13}}, \quad D_k^{(1)}=a_1(k+2)^{3/13},\\ D_k^{(2)}&=a_2(k+2)^{3/13},
		\end{align*}
		with 
		\begin{align*}
			0&< a_1\le\min\{\frac{b}{32a_4L^2},\frac{ba_4}{2L\beta}\},\quad 0<a_2\le\min\{\frac{b}{32a_4L^2},\frac{ba_4}{2L\beta}\},\\
			0&<a_4\le\min\{\frac{1}{8L},\frac{\beta}{8\sqrt{5}}\},\quad a_5\ge\frac{4a_4}{a_1}+\frac{12}{13},\quad a_6\ge\frac{80a_4}{a_2}+\frac{12}{13}.
		\end{align*}
		If $0<\beta\le\frac{1}{6L}$, then for any given $\varepsilon>0$, $$	T( \varepsilon)\le \max\{\tilde{T}(\varepsilon),(\frac{2L\sigma_y}{\varepsilon})^{\frac{13}{2}}-1\},$$
		where $\tilde{T}(\varepsilon)$ satisfies that  
		\begin{equation}
			\frac{\varepsilon^2}{4}\le\frac{C_1+C_2\ln(\tilde{T}(\varepsilon)+2)}{d_1a_4(\frac{13}{4}(\tilde{T}(\varepsilon)+3)^{4/13}-\frac{13\cdot3^{4/13}}{4})},\label{thm2.1:1}
		\end{equation}
		with
		\begin{align*}
			C_1&=S_{1}(x_{1},y_{1})-\underbar{S}+\frac{\rho_1}{2}\sigma_y^2+\frac{40\alpha_1L^2\rho_1}{\eta_{1}\beta^2\rho_2^3}\sigma_y^2,\quad d_1\le\min\{\frac{1}{8},L\beta,\frac{a_1a_5}{4a_4},\frac{a_2a_6}{4a_4}\},\\
			C_2&=\frac{2\delta^2(a_5^2a_1+a_6^2a_2)}{b}+\frac{12a_2L^2\sigma_y^2}{13b},\quad\underbar{S}=\min\limits_{x\in\mathcal{X}}\min\limits_{y\in\mathcal{Y}}S_k(x,y),\\
			\sigma_y&=\max\{\|y\| \mid y\in\mathcal{Y}\}.
		\end{align*}
	\end{thm}
	
	\begin{proof}
		By the definition of $\alpha_k$, $\rho_k$ and $D_k^{(1)}$, we  have $\frac{\alpha_k}{\rho_{k+1}}\le\frac{\alpha_{k-1}}{\rho_{k}}$ and
		\begin{align*}
			D_k^{(1)}-D_{k-1}^{(1)}&\le\frac{3a_1(k+1)^{-10/13}}{13}\le\frac{3a_1\cdot2^{10/13}(2+k)^{-10/13}}{13}\nonumber\\
			&\le\frac{6a_1(2+k)^{-9/13}}{13}.
		\end{align*}
		Then, by the setting of $a_5$, we get
		\begin{align*}
			2\eta_k\alpha_k-D_k^{(1)}\gamma_{k}+D_k^{(1)}-D_{k-1}^{(1)}&\le(2a_4-a_1a_5+\frac{6a_1}{13})(2+k)^{-9/13}\nonumber\\
			&\le-\frac{a_1a_5}{2}(2+k)^{-9/13}.
		\end{align*}
		Similarly, we also have
		\begin{align*}
			\frac{40\eta_k\alpha_kL^2}{\rho_{k+1}^2}-D_k^{(2)}\theta_{k}+D_k^{(2)}-D_{k-1}^{(2)}&\le(40a_4-a_2a_6+\frac{6a_2}{13})(2+k)^{-5/13}\nonumber\\
			&\le-\frac{a_2a_6}{2}(2+k)^{-5/13}.
		\end{align*}
		By the settings of $a_1, a_2, a_4$, we obtain
		\begin{align*}
			&-(\frac{3\eta_k}{4\alpha_k}-\frac{L^2\eta_k^2}{\rho_{k+1}}-\frac{40L^4\eta_k\alpha_k}{\rho_{k+1}^4\beta^2}-\frac{4L^2\eta_k^2(D_k^{(1)}d_x+D_k^{(2)}d_y)}{b})\\
			\le&(-\frac{3}{4a_4}+L+\frac{40a_4}{\beta^2}+\frac{4L^2(a_1d_x+a_2d_y)}{b})(2+k)^{-1/13}\nonumber\\
			\le&-\frac{1}{4a_4}(2+k)^{-1/13},
		\end{align*}
		and
		\begin{align*}
			&-(\frac{6\alpha_kL^2\eta_k}{\rho_{k+1}\beta}-\frac{4L^2\eta_k^2(D_k^{(1)}d_x+D_k^{(2)}d_y)}{b})\\
			\le&(-\frac{6La_4}{\beta}+\frac{4L^2(a_1d_x+a_2d_y)}{b})(2+k)^{-7/13}
			\le-\frac{2La_4}{\beta}(2+k)^{-7/13}.
		\end{align*}
		Plugging these inequalities into \eqref{safocthm1:2}, we get
		\begin{align}\label{safocthm1:3}
			&S_{k+1}(x_{k+1},y_{k+1})-S_{k}(x_{k},y_{k})\nonumber\\
			\le&-\frac{a_1a_5}{2}(2+k)^{-9/13}\mathbb{E}[\|\nabla _{x}g_k(x_{k},y_{k})-v_k\|^2]\nonumber\\
			&-\frac{a_2a_6}{2}(2+k)^{-5/13}\mathbb{E}[\|\nabla _{y}g_k(x_{k},y_{k})-w_k\|^2]\nonumber\\
			&-\frac{1}{4a_4}(2+k)^{-1/13}\mathbb{E}[\|\tilde{x}_{k+1}-x_k\|^2]-\frac{2La_4}{\beta}(2+k)^{-7/13}\mathbb{E}[\|\tilde{y}_{k+1} -y_k\|^2]\nonumber\\
			&+(\frac{40\alpha_kL^2(\rho_k-\rho_{k+1})}{\eta_{k}\beta^2\rho_{k+1}^3}-\frac{40\alpha_{k+1}L^2(\rho_{k+1}-\rho_{k+2})}{\eta_{k+1}\beta^2\rho_{k+1}^3})\sigma_y^2+\frac{\rho_k-\rho_{k+1}}{2}\sigma_y^2\nonumber\\
			&+\frac{2\delta^2(\gamma_k^2D_k^{(1)}+\theta_k^2D_k^{(2)})}{b}+\frac{12a_2L^2\sigma_y^2}{13b}(k+2)^{-1}.
		\end{align}
		It follows from the definition of $d_1$, \eqref{safoclem5:1} and \eqref{safocthm1:3} that $\forall k\ge 1$,
		\begin{align}\label{safocthm1:4}
			&d_1\eta_k\alpha_k\mathbb{E}[\|\nabla \tilde{\mathcal{G}}_k^{\alpha_k,\beta}(x_k,y_k)\|^2]\nonumber\\
			\le&S_{k}(x_{k},y_{k})-S_{k+1}(x_{k+1},y_{k+1})+\frac{\rho_k-\rho_{k+1}}{2}\sigma_y^2\nonumber\\
			&+(\frac{40\alpha_kL^2(\rho_k-\rho_{k+1})}{\eta_{k}\beta^2\rho_{k+1}^3}-\frac{40\alpha_{k+1}L^2(\rho_{k+1}-\rho_{k+2})}{\eta_{k+1}\beta^2\rho_{k+1}^3})\sigma_y^2\nonumber\\
			&+\frac{2\delta^2(\gamma_k^2D_k^{(1)}+\theta_k^2D_k^{(2)})}{b}+\frac{12a_2L^2\sigma_y^2}{13b}(k+2)^{-1}.
		\end{align}
		Denoting $\underbar{S}=\min\limits_{x\in\mathcal{X}}\min\limits_{y\in\mathcal{Y}}S_k(x,y)$, $\tilde{T}(\varepsilon)=\min\{k \mid \| \nabla \tilde{\mathcal{G}}_k^{\alpha_k,\beta}(x_k,y_{k}) \| \leq \frac{\varepsilon}{2}, k\geq 1\}$. By summing both sides of \eqref{safocthm1:4} from $k=1$ to $\tilde{T}(\varepsilon)$, we obtain
		\begin{align}
			&\sum_{k=1}^{\tilde{T}(\varepsilon)}d_1\eta_k\alpha_k\mathbb{E}[\|\nabla \tilde{\mathcal{G}}_k^{\alpha_k,\beta}(x_k,y_k)\|^2]\nonumber\\
			\le& S_{1}(x_{1},y_{1})-S_{\tilde{T}(\varepsilon)+1}(x_{\tilde{T}(\varepsilon)+1},y_{\tilde{T}(\varepsilon)+1})+\frac{\rho_1}{2}\sigma_y^2+\sum_{k=1}^{\tilde{T}(\varepsilon)}\frac{2\delta^2(\gamma_k^2D_k^{(1)}+\theta_k^2D_k^{(2)})}{b}\nonumber\\
			&+\frac{40\alpha_1L^2\rho_1\sigma_y^2}{\eta_{1}\beta^2\rho_2^3}+\sum_{k=1}^{\tilde{T}(\varepsilon)}\frac{12a_2d_yL^2\sigma_y^2}{13b}(k+2)^{-1}\nonumber\\
			\le& S_{1}(x_{1},y_{1})-\underbar{S}+\frac{\rho_1}{2}\sigma_y^2+\frac{40\alpha_1L^2\rho_1\sigma_y^2}{\eta_{1}\beta^2\rho_2^3}+\sum_{k=1}^{\tilde{T}(\varepsilon)}\frac{2\delta^2(a_5^2a_1+a_6^2a_2)}{b}(k+2)^{-1}\nonumber\\
			&+\sum_{k=1}^{\tilde{T}(\varepsilon)}\frac{12a_2L^2\sigma_y^2}{13b}(k+2)^{-1}.\label{safocthm1:5}
		\end{align}
		Since $\sum_{k=1}^{\tilde{T}(\varepsilon)}(k+2)^{-1}\le\ln(\tilde{T}(\varepsilon)+2)$ and $\sum_{k=1}^{\tilde{T}(\varepsilon)}(k+2)^{-9/13}\ge\frac{13}{4}(\tilde{T}(\varepsilon)+3)^{4/13}-\frac{13\cdot3^{4/13}}{4}$, by the definition of $C_1$ and $C_2$, we get
		\begin{align}
			\frac{\varepsilon^2}{4}\le\frac{C_1+C_2\ln(\tilde{T}(\varepsilon)+2)}{d_1a_4(\frac{13}{4}(\tilde{T}(\varepsilon)+3)^{4/13}-\frac{13\cdot3^{4/13}}{4})}.\label{safocthm1:7}
		\end{align}
		On the other hand, if $k\ge(\frac{2L\sigma_y}{\varepsilon})^{\frac{13}{2}}-1$, then $\rho_k\le\frac{\varepsilon}{2\sigma_y}$. This inequality together with the definition of $\sigma_y$ then imply that $\rho_k\|y_k\|\le\frac{\varepsilon}{2}$. Therefore, there exists a
			\begin{align*}
				T( \varepsilon)\le \max\{\tilde{T}(\varepsilon),(\frac{2L\sigma_y}{\varepsilon})^{\frac{13}{2}}-1\}.
			\end{align*}
			such that $\mathbb{E}[\|\nabla \mathcal{G}^{\alpha_k,\beta}(x_k,y_k)]\|\le\mathbb{E}[\|\nabla \tilde{\mathcal{G}}_k^{\alpha_k,\beta}(x_k,y_k)\|]+\rho_{k}\|y_{k}\|\le\varepsilon$ which completes the proof. 
		\end{proof}

		\begin{re}
			It is easily verified from \eqref{thm2.1:1} that $\mathcal{O} (\varepsilon ^2) =\mathcal{O}\left(\frac{\ln\tilde{T}(\varepsilon)}{\tilde{T}(\varepsilon)^{4/13}}\right)=\tilde{\mathcal{O}}(\tilde{T}(\varepsilon)^{-4/13})$, which means $\tilde{T}(\varepsilon)=\tilde{\mathcal{O}}\left(\varepsilon ^{-6.5} \right)$. Therefore, $T(\varepsilon)=\tilde{\mathcal{O}}\left(\varepsilon ^{-6.5} \right)$ by Theorem \ref{safocthm1}, which means that the number of iterations for Algorithm \ref{sfoalg:1} to obtain an $\varepsilon$-stationary point of problem \eqref{prob-s} is upper bounded by $\tilde{\mathcal{O}}\left(\varepsilon ^{-6.5} \right)$ for solving stochastic nonconvex-concave minimax problems.
		\end{re}

		\subsection{Nonsmooth Case.}
		Consider the following stochastic nonsmooth nonconvex-concave minimax optimization problem:
		\begin{equation}\label{prob-sn}
			\min \limits_{x\in \mathcal{X}}\max \limits_{y\in \mathcal{Y}}g(x,y)+f(x)-h(y),
		\end{equation}
		where $g(x,y)=\mathbb{E}_{\zeta\sim D}[G(x,y,\zeta)]$, both $f(x)$ and $h(y)$ are convex continuous but possibly nonsmooth function. 
		
		Denote the proximity operator as 
		$\operatorname{Prox}_{h,\mathcal{Z}}^{1/\alpha}(\upsilon):=\arg\min\limits_{z\in \mathcal{Z}} h(z)+\frac{\alpha}{2}\|z-\upsilon \|^2$. Based on FORMDA Algorithm, we propse an accelerated first-order regularized  momentum descent ascent algorithm for nonsmooth minimax problem (FORMDA-NS) by replacing $\mathcal{P}_{\mathcal{X}}^{1/\alpha_k}$ with $\operatorname{Prox}_{f,\mathcal{X}}^{1/\alpha_k}$, $\mathcal{P}_{\mathcal{Y}}^{1/\beta}$ with $\operatorname{Prox}_{h,\mathcal{Y}}^{1/\beta}$ in FORMDA Algorithm. The proposed FORMDA-NS algorithm is formally stated in Algorithm \ref{sfoalg:2}.  

		\begin{algorithm}[t]
			\caption{An Accelerated First-order Regularized  Momentum Descent Ascent Algorithm for Nonsmooth Minimax Problem (FORMDA-NS)}
			\label{sfoalg:2}
			\begin{algorithmic}
				\State{\textbf{Step 1}:Input $x_1,y_1,\lambda_1,\alpha_1,\beta,0<\eta_1\leq 1, b$; $\gamma_0=1$,$\theta_0=1$. Set $k=1$. }
				\State{\textbf{Step 2}: Draw a mini-batch samples $I_{k}=\{\zeta_i^{k}\}_{i=1}^b$. Compute 
					\begin{align}
						v_{k}&=\nabla_x\tilde{G}_{k}( x_{k},y_{k};I_{k})+(1-\gamma_{k-1})[v_{k-1}-\nabla _x\tilde{G}_{k-1}(x_{k-1},y_{k-1};I_{k})]\label{vk-ns}
					\end{align}			
					and 
					\begin{align}
						w_{k}&=\nabla_y\tilde{G}_{k}( x_{k},y_{k};I_{k})+(1-\theta_{k-1})[w_{k-1}-\nabla_y\tilde{G}_{k-1}( x_{k-1},y_{k-1};I_{k})]\label{wk-ns}
				\end{align}}
			\State{\textbf{Step 3}:Perform the following update for $x_k$ and $y_k$:  	
				\qquad \begin{align}
					\tilde{x}_{k+1}&=\operatorname{Prox}_{f,\mathcal{X}}^{1/\alpha_k} \left( x_k - \alpha_kv_k\right),\label{sfo2:update-tx}\\ x_{k+1}&=x_k+\eta_k(\tilde{x}_{k+1}-x_k),\label{sfo2:update-x}\\
					\tilde{y}_{k+1}&=\operatorname{Prox}_{h,\mathcal{Y}}^{1/\beta} \left( y_k + \beta w_k\right),\label{sfo2:update-ty}\\ y_{k+1}&=y_k+\eta_k(\tilde{y}_{k+1}-y_k).\label{sfo2:update-y}
			\end{align}}
			\State{\textbf{Step 4}:If some stationary condition is satisfied, stop; otherwise, set $k=k+1, $ go to Step 2.}
		\end{algorithmic}
	\end{algorithm}
	
	Before analyzing the convergence of Algorithm \ref{sfoalg:2}, we define the stationarity gap as the termination criterion as follows.
	\begin{defi}\label{azogap-fn}
		For some given $\alpha_k>0$ and $\beta>0$, the stationarity gap for problem \eqref{prob-sn} is defined as 
		\begin{equation*}
			\nabla \mathcal{G}^{\alpha_k,\beta}( x,y) :=\left( \begin{array}{c}
				\frac{1}{\alpha_k}\left( x-\operatorname{Prox}_{f,\mathcal{X}}^{\frac{1}{\alpha_k}}\left( x-\alpha_k\nabla _xg( x,y) \right) \right)\\
				\frac{1}{\beta} \left( y-\operatorname{Prox}_{h,\mathcal{Y}}^{\frac{1}{\beta} }\left( y+ \beta\nabla _yg( x,y) \right) \right)
			\end{array} \right) .
		\end{equation*}
	\end{defi}
	
	Denote
	\begin{align}
		\Phi_k(x)&:=\max_{y\in\mathcal{Y}}(g_k(x,y)-h(y)),\label{snphik}\\
		y_k^*(x)&:=\arg\max_{y\in\mathcal{Y}}(g_k(x,y)-h(y)).\label{snyk*}
	\end{align}
	By Proposition B.25 in \cite{Bertsekas}, we have $\nabla_x\Phi_k(x)=\nabla_x g_k(x, y_k^*(x))$ and $\Phi_k(x)$ is $L_{\Phi_k}$-Lipschitz smooth with $L_{\Phi_k}=L+\frac{L^2}{\rho_k}$ under the $\rho_k$-strong concavity of $g_k(x,y)$. Similar to the proof of Section 2.1, we first estimate an upper bound of $\|y_{k+1}^*(\bar{x})-y_{k}^*(x)\|^2$.

	\begin{lem}\label{snlem1}
		Suppose that Assumptions \ref{azoass:Lip} and \ref{sazocrho} hold. Then for any $x, \bar{x}\in\mathcal{X}$,
		\begin{align}\label{snlem1:1}
			\|y_{k+1}^*(\bar{x})-y_{k}^*(x)\|^2
			\le\frac{L^2}{\rho_{k+1}^2}\|\bar{x}-x\|^2+\frac{\rho_k-\rho_{k+1}}{\rho_{k+1}}(\|y_{k+1}^*(\bar{x})\|^2-\|y_{k}^*(x)\|^2).
		\end{align}
	\end{lem}
	\begin{proof}
		The optimality condition for $y_k^*(x)$ implies that $\forall y\in \mathcal{Y}$ and $\forall k\geq 1$,
		\begin{align}
			\langle\nabla_yg_{k+1}(\bar{x},y_{k+1}^*(\bar{x}))-\partial h(y_{k+1}^*(\bar{x})), y_{k}^*(x)-y_{k+1}^*(\bar{x})\rangle&\le0,\label{snlem1:2}\\
			\langle\nabla_yg_k(x,y_{k}^*(x))-\partial h(y_{k}^*(x)), y_{k+1}^*(\bar{x})-y_{k}^*(x)\rangle&\le0.\label{snlem1:3}
		\end{align}
		By $\langle\partial h(y_{k+1}^*(\bar{x}))-\partial h(y_{k}^*(x)), y_{k+1}^*(\bar{x})-y_{k}^*(x)\rangle\ge0$, and similar to the proof of Lemma 2.1, we complete the proof.
	\end{proof}
	
	Next, we prove the iteration complexity of Algorithm \ref{sfoalg:2}.
	\begin{lem}\label{snlem2}
		Suppose that Assumptions \ref{azoass:Lip} and \ref{sazocrho} hold. Let $\{(x_k,y_k)\}$ be a sequence generated by Algorithm \ref{sfoalg:2}, if $\rho_k\le L$ and $0<\eta_k\le1$, then $\forall k \ge 1$, 
		\begin{align}
			&\Phi_{k+1}(x_{k+1})+f(x_{k+1})-\Phi_k(x_{k})-f(x_k)\nonumber\\
			\le&2\eta_k\alpha_k L^2\|y_k-y_k^*(x_k)\|^2+2\eta_k\alpha_k\|\nabla _xg_k(x_{k},y_{k})-v_k\|^2\nonumber\\
			&-(\frac{3\eta_k}{4\alpha_k}-\frac{L^2\eta_k^2}{\rho_{k+1}})\|\tilde{x}_{k+1}-x_k\|^2+\frac{\rho_k-\rho_{k+1}}{2}\sigma_y^2,\label{snlem2:1}
		\end{align}
		where $\sigma_y=\max\{\|y\| \mid y\in\mathcal{Y}\}$.
	\end{lem}
	\begin{proof}
		The optimality condition for $x_k$ in \eqref{sfo2:update-tx} implies that $\forall x\in \mathcal{X}$ and $\forall k\geq 1$,
		\begin{equation}\label{snlem2:2}
			\langle v_k,\tilde{x}_{k+1}-x_k \rangle \le -\frac{1}{\alpha_k}\|\tilde{x}_{k+1}-x_k\|^2-\langle \xi_{k+1},\tilde{x}_{k+1}-x_k \rangle, 
		\end{equation}
		where $\xi_{k+1}\in\partial f(\tilde{x}_{k+1})$.
		Similar to the proof of Lemma \ref{safoclem2}, by replacing \eqref{safoclem2:5} in Lemma \ref{safoclem2} with \eqref{snlem2:2}, we have 
		\begin{align}
			\Phi_{k+1}(x_{k+1})-\Phi_k(x_{k})
			\le&2\eta_k\alpha_k L^2\|y_k-y_k^*(x_k)\|^2+2\eta_k\alpha_k\|\nabla _xg_k(x_{k},y_{k})-v_k\|^2\nonumber\\
			&-(\frac{3\eta_k}{4\alpha_k}-\frac{L^2\eta_k^2}{\rho_{k+1}})\|\tilde{x}_{k+1}-x_k\|^2+\frac{\rho_k-\rho_{k+1}}{2}\sigma_y^2\nonumber\\
			&-\eta_k\langle \xi_{k+1},\tilde{x}_{k+1}-x_k \rangle .\label{snlem2:3}
		\end{align}
		By the definition of subgradient, we have $f(x_{k})-f( \tilde{x}_{k+1})\ge\langle\xi_{k+1},x_{k}-\tilde{x}_{k+1} \rangle$, then we have 
		\begin{align}\label{snlem2:4}
			-\eta_k\langle \xi_{k+1},\tilde{x}_{k+1}-x_k \rangle \le\eta_k(f(x_{k})-f( \tilde{x}_{k+1})).
		\end{align}      
		Utilizing the convexity of $f(x)$, $x_{k+1}=x_k+\eta_k(\tilde{x}_{k+1}-x_k)$ and $0<\eta_k\le1$, we get
		\begin{align}\label{snlem2:5}
			f(x_{k+1})\le(1-\eta_k)f(x_k)+\eta_kf(\tilde{x}_{k+1}).
		\end{align}     
		The proof is completed by combining \eqref{snlem2:3}, \eqref{snlem2:4} and \eqref{snlem2:5}.
	\end{proof}

	\begin{lem}\label{snlem3}
		Suppose that Assumptions \ref{azoass:Lip} and \ref{sazocrho} hold. Let $\{(x_k,y_k)\}$ be a sequence generated by Algorithm \ref{sfoalg:2}, if $0<\eta_k\le1$, $0<\beta\le\frac{1}{6L}$ and $\rho_k\le L$, then $\forall k \ge 1$, 
		\begin{align}
			&\|y_{k+1}-y_{k+1}^*(x_{k+1})\|^2\nonumber\\
			\le&(1-\frac{\eta_k\beta\rho_{k+1}}{4})\|y_{k}-y_k^*(x_k)\|^2-\frac{3\eta_k}{4}\|\tilde{y}_{k+1} -y_k\|^2+\frac{5L^2\eta_k}{\rho_{k+1}^3\beta}\|\tilde{x}_{k+1}-x_k\|^2\nonumber\\
			&+\frac{5\eta_k\beta}{\rho_{k+1}}\|\nabla_{y}g_k(x_{k},y_{k})-w_k\|^2+\frac{5(\rho_k-\rho_{k+1})}{\rho_{k+1}^2\eta_k\beta}(\|y_{k+1}^*(x_{k+1})\|^2-\|y_{k}^*(x_k)\|^2).\label{snlem3:1}
		\end{align}
	\end{lem}
	\begin{proof}
		The optimality condition for $y_k$ in \eqref{sfo2:update-ty} implies that $\forall y\in \mathcal{Y}$ and $\forall k\geq 1$,
		\begin{align}\label{snlem3:2}
			\langle w_k,y-\tilde{y}_{k+1} \rangle &\le \frac{1}{\beta}\langle\tilde{y}_{k+1}-y_k,y-\tilde{y}_{k+1}\rangle+\langle \varrho_{k+1},y-\tilde{y}_{k+1} \rangle, \nonumber\\
			&=-\frac{1}{\beta}\|\tilde{y}_{k+1}-y_k\|^2+\frac{1}{\beta}\langle\tilde{y}_{k+1}-y_k,y-y_k \rangle+\langle \varrho_{k+1},y-\tilde{y}_{k+1} \rangle\nonumber\\
			&\le-\frac{1}{\beta}\|\tilde{y}_{k+1}-y_k\|^2+\frac{1}{\beta}\langle\tilde{y}_{k+1}-y_k,y-y_k \rangle+h(y)-h(\tilde{y}_{k+1}),
		\end{align}
		where $\varrho_{k+1}\in\partial h(\tilde{y}_{k+1})$, the last inequality holds by $h(y)-h( \tilde{y}_{k+1})\ge\langle\varrho_{k+1},y-\tilde{y}_{k+1} \rangle$. Similar to the proof of Lemma \ref{safoclem3}, by replacing \eqref{safoclem3:4} in Lemma \ref{safoclem3} with \eqref{snlem3:2}, we complete the proof.
	\end{proof}

	\begin{thm}\label{thm1}
		Suppose that Assumptions \ref{azoass:Lip}, \ref{sazocrho} and \ref{afoass:var} hold. Let $\{\left(x_k,y_k\right)\}$ be a sequence generated by Algorithm \ref{sfoalg:2}. 
		For any given $k\geq 1$, let
		\begin{align*}
			\eta_k&=\frac{1}{(k+2)^{5/13}},\quad \alpha_k=\frac{a_4}{(k+2)^{4/13}},\quad \rho_{k}=\frac{L}{(k+1)^{2/13}},\\
			\gamma_k&=\frac{a_5}{(k+2)^{12/13}},\quad \theta_k=\frac{a_6}{(k+2)^{8/13}},
		\end{align*}
		with 
		\begin{align*}
			0&< a_1\le\min\{\frac{b}{32a_4L^2},\frac{ba_4}{2L\beta}\},\quad 0<a_2\le\min\{\frac{b}{32a_4L^2},\frac{ba_4}{2L\beta}\},\\
			0&<a_4\le\min\{\frac{1}{8L},\frac{\beta}{8\sqrt{5}}\},\quad a_5\ge\frac{4a_4}{a_1}+\frac{12}{13},\quad a_6\ge\frac{80a_4}{a_2}+\frac{12}{13}.
		\end{align*}
		If $0<\beta\le\frac{1}{6L}$, then for any given $\varepsilon>0$, $$	T( \varepsilon)\le \max\{\tilde{T}(\varepsilon),(\frac{2L\sigma_y}{\varepsilon})^{\frac{13}{2}}-1\},$$
		where $\tilde{T}(\varepsilon)$ satisfies that  
		\begin{equation}
			\frac{\varepsilon^2}{4}\le\frac{C_1+C_2\ln(\tilde{T}(\varepsilon)+2)}{d_1a_4(\frac{13}{4}(\tilde{T}(\varepsilon)+3)^{4/13}-\frac{13\cdot3^{4/13}}{4})},\label{snthm2.1:1}
		\end{equation}
		with
		\begin{align*}
			C_1&=S_{1}(x_{1},y_{1})-\underbar{S}+\frac{\rho_1}{2}\sigma_y^2+\frac{40\alpha_1L^2\rho_1}{\eta_{1}\beta^2\rho_2^3}\sigma_y^2,\quad d_1\le\min\{\frac{1}{8},L\beta,\frac{a_1a_5}{4a_4},\frac{a_2a_6}{4a_4}\},\\
			C_2&=\frac{2\delta^2(a_5^2a_1+a_6^2a_2)}{b}+\frac{12a_2L^2\sigma_y^2}{13b},\quad\underbar{S}=\min\limits_{x\in\mathcal{X}}\min\limits_{y\in\mathcal{Y}}S_k(x,y),\\
			\sigma_y&=\max\{\|y\| \mid y\in\mathcal{Y}\}.
		\end{align*}
	\end{thm}
	
	\begin{proof}
		Similar to the proof of Lemma \ref{safoclem4}, \ref{safoclem45}, \ref{safoclem5} and Theorem \ref{safocthm1}, by replacing $F_{k+1}(x_{k+1},y_{k+1})$ with $\tilde{F}_{k+1}(x_{k+1},y_{k+1}):=F_{k+1}(x_{k+1},y_{k+1})+\mathbb{E}[f(x_{k+1})]$, $\mathcal{P}_{\mathcal{X}}^{1/\alpha_k}$ with $\operatorname{Prox}_{f,\mathcal{X}}^{1/\alpha_k}$, $\mathcal{P}_{\mathcal{Y}}^{1/\beta}$ with $\operatorname{Prox}_{h,\mathcal{Y}}^{1/\beta}$, respectively, we complete the proof.
	\end{proof}
	
	Theorem \ref{thm1} shows that the iteration complexity of Algorithm \ref{sfoalg:2} to obtain an $\varepsilon$-stationary point is upper bounded by $\tilde{\mathcal{O}}(\varepsilon ^{-6.5})$ under the nonsmooth setting.

	\section{Numerical Results}\label{sec4}
	In this section, we compare the numerical performance of  the proposed FORMDA algorithm with some existing algorithms for solving two classes of problems. All the numerical tests are implemented in Python 3.9. The first test is run in a laptopwith 2.30GHz processor and the second one is carried out on a RTX 4090 GPU.
	
	Firstly, we consider the following Wasserstein GAN (WGAN) problem \cite{Arjovsky},
	\begin{align*}
		\min _{\varphi_1, \varphi_2} \max _{\phi_1, \phi_2} f\left(\varphi_1, \varphi_2, \phi_1, \phi_2\right) \triangleq \mathbb{E}_{\left(x^{real}, z\right) \sim \mathcal{D}}\left[D_\phi\left(x^{real}\right)-D_\phi\left(G_{\varphi_1, \varphi_2}(z)\right)\right],
	\end{align*}
	where $G_{\varphi_1, \varphi_2}(z)=\varphi_1+\varphi_2 z, D_\phi(x)=\phi_1 x+\phi_2 x^2, \phi=\left(\phi_1, \phi_2\right), x^{r e a l}$ is a normally distributed random variable with mean $\varphi_1^*=0$ and variance $\varphi_2^*=0.1$, and $z$ a normally distributed random variable with mean $\phi_1^*=0$ and variance $\phi_2^*=1$. The optimal solution is $(\varphi_1^*, \varphi_2^*, \phi_1^*,\phi_2^*)=(0,0.1,0,1)$.
	
	We compare the numerical performance of the proposed FORMDA algorithm with the SGDA algorithm~\cite{Lin2019} and the PG-SMD algorithm \cite{Rafique} for solving the WGAN problem.  The batch size is set to be $b=100$ for all three tested algorithms. The parameters in FORMDA are chosen to be $\eta_k=\frac{1}{(k+2)^{5/13}}$, $\alpha_k=\frac{0.5}{(k+2)^{4/13}}$, $\rho_{k}=\frac{1}{(k+1)^{2/13}}$, $\gamma_k=\frac{3}{(k+2)^{12/13}}$, $\theta_k=\frac{2}{(k+2)^{8/13}}$, $\beta=0.005$.
	All the parameters of the PG-SMD algorithm and the SGDA algorithm are chosen the same as that in \cite{Rafique} and  \cite{Lin2019} respectively.
	
	\begin{figure}[htb]
		\centering  
		\subfigure{
			\includegraphics[width=0.3\textwidth]{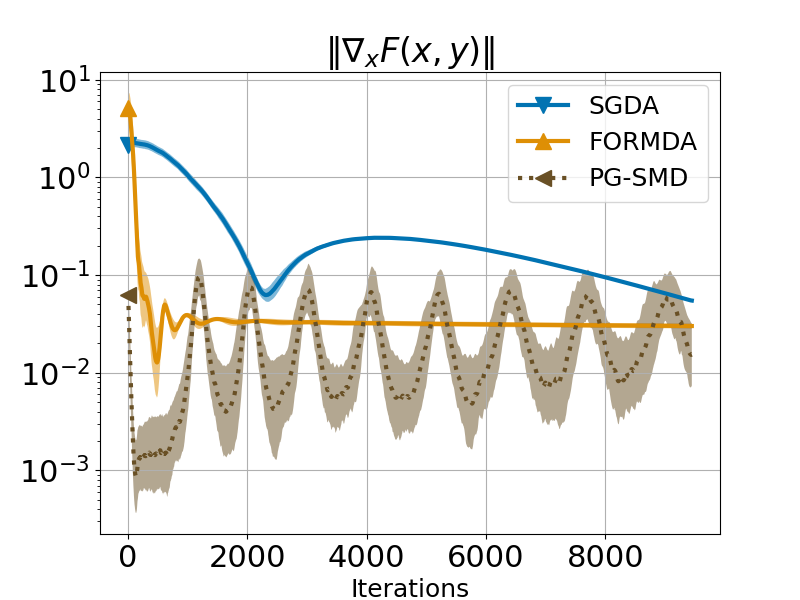}}
		\subfigure{
			\includegraphics[width=0.3\textwidth]{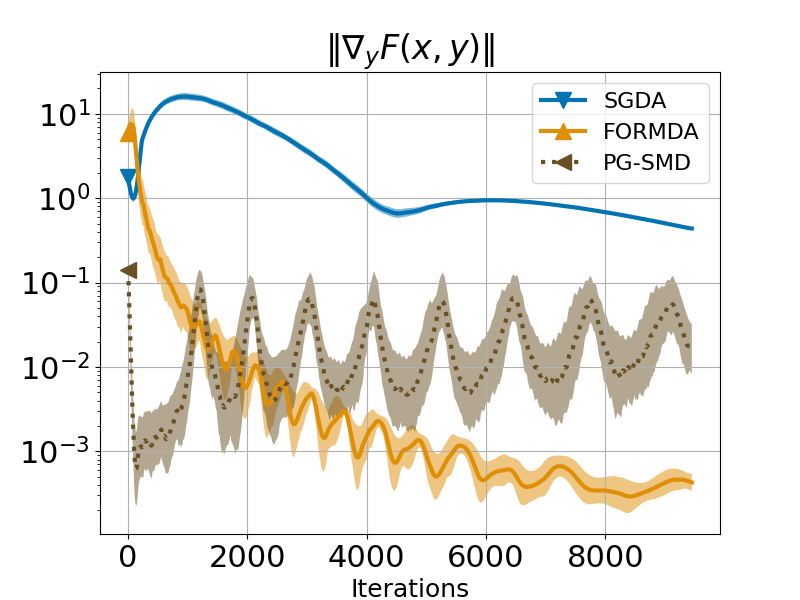}}
		\subfigure{
			\includegraphics[width=0.3\textwidth]{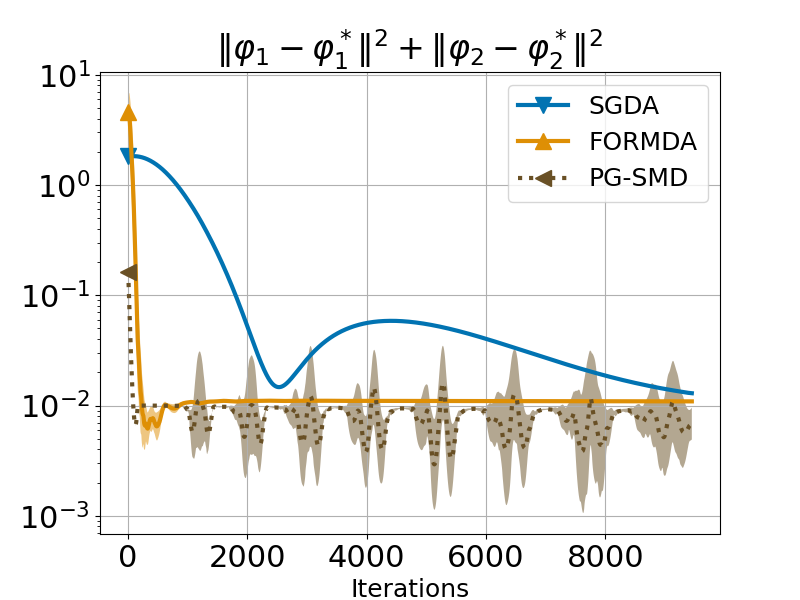}}
		\caption{ Performance of the PG-SMD algorithm, the SGDA algorithm and  the FORMDA algorithm for WGAN problem.}
		\label{Fig1}
	\end{figure}
	
	Figure \ref{Fig1} shows the average change in the stochastic gradient norm and the distance from the iteration point to the optimal value point for 5 independent runs of the four test algorithms, and the shaded area around the line indicates the standard deviation. It can be found that the proposed FORMDA algorithm is better than SGDA and close to the performance of PG-SMD.
	
	
	Secondly, we consider a robust learning problem over multiple domains \cite{Qian}.
	Let $P_1, \ldots, P_M$ denote $M$ distributions of data. Robust learning over multiple domains is to minimize the maximum of the expected losses over the $M$ distributions, it can be formulated as the following stochastic nonconvex-concave minimax problem:
	\begin{equation*}
		\min _{x \in \mathcal{X}} \max _m \mathbb{E}_{\zeta \sim P_m}[F(x;\zeta)]=\min _{x} \max _{y \in \Delta} \sum_{m=1}^{M}y_m f_m(x),
	\end{equation*}
	where $f_m(x)=\mathbb{E}_{\zeta \sim P_m}[F(x; \zeta)]$, $F(x;\zeta)$ is a nonnegative loss function,  $ x $ denotes the network parameters, $ y \in \Delta$ describes the weights over different tasks and $ \Delta $ is the simplex, i.e., $\Delta=\left\{(y_1,\cdots,y_M)^{T} \mid \sum_{m=1}^{M} y_{m}=1,  y_{m} \geq 0\right\}$.	
	We consider two image classification problems with MNIST and CIFAR10 datasets. Our goal is to train a neural network that works on these two completely unrelated problems simultaneously. The quality of the algorithm for solving a robust learning problem over multiple domains is measured by the worst case accuracy over all domains \cite{Qian}.
	
	We compare the numerical performance of the proposed FORMDA algorithm with SGDA algorithm~\cite{Lin2019} and PG-SMD algorithm \cite{Rafique}. We use cross-entropy loss as our loss function and set the batch size in all tasks to be 32. All algorithms were run for 50 epochs. The parameters in FORMDA are chosen to be $\eta_k=\frac{1}{(k+12)^{5/13}+1}$, $\alpha_k=\frac{0.5}{(k+12)^{2/13}+1}$, $\rho_{k}=\frac{8}{(k+11)^{2/13}+2}$, $\gamma_k=\frac{3}{(k+12)^{8/13}+1}$, $\theta_k=\frac{2}{(k+12)^{8/13}+1}$, $\beta=0.001$.
	All the parameters of the PG-SMD algorithm and the SGDA algorithm are chosen the same as that in \cite{Rafique} and  \cite{Lin2019} respectively.
	
	\begin{figure}[htb]
		\centering  
		\subfigure{
			\includegraphics[width=0.4\textwidth]{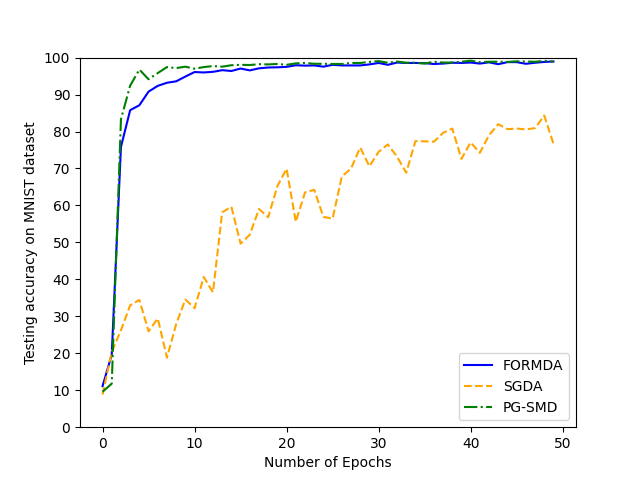}}
		\subfigure{
			\includegraphics[width=0.4\textwidth]{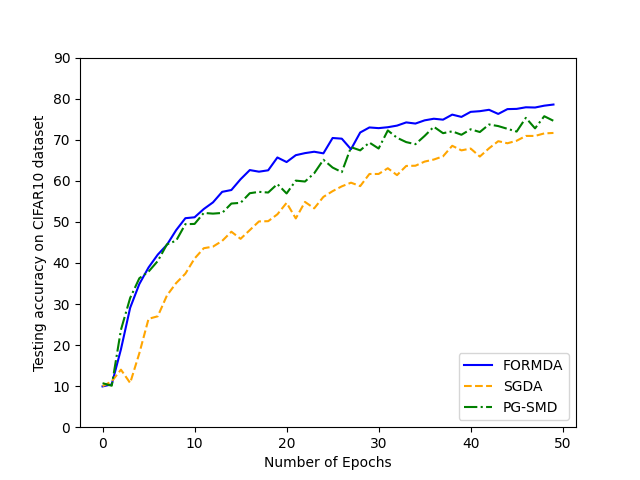}}
		\caption{Performance of three algorithms for solving robust multi-task learning problem.}
		\label{Fig2}
	\end{figure}
	
	Figure \ref{Fig2} shows the testing accuracy on MNIST and CIFAR10 datasets. It can be found that the proposed FORMDA algorithm outperforms SGDA and has similar performance to PG-SMD.

	\section{Conclusions}
	In this paper, we propose an accelerated FORMDA algorithm for solving stochastic nonconvex-concave minimax problems. The iteration complexity of the algorithm is proved to be $\tilde{\mathcal{O}}(\varepsilon ^{-6.5})$ to obtain an $\varepsilon$-stationary point. 
	It owns the optimal complexity bounds within single-loop algorithms for solving stochastic nonconvex-concave minimax problems till now. Moreover, we can extend the analysis to the scenarios with possibly nonsmooth objective functions, and obtain that the iteration complexity is $\tilde{\mathcal{O}}(\varepsilon ^{-6.5})$.
	Numerical experiments show the efficiency of the proposed algorithm. Whether there is a single-loop algorithm with better complexity for solving this type of problem is still worthy of further research.

	\bmhead{Acknowledgments}
	The authors are supported by National Natural Science Foundation of China under the Grant (Nos. 12071279 and 12471294).
	%
	%
	
	\section*{Declarations}
	%
	%
	\begin{itemize}
		\item {\bf Conflict of interest} The authors have no relevant financial or non-financial interests to disclose.
		\item {\bf Data availability} Data sharing not applicable to this article as no datasets were generated or analyzed during the current study.
	\end{itemize}
	%
	%
	%
	%
	%
	%

	
	
	
\end{document}